\documentclass{amsart}
\usepackage[utf8]{inputenc}
\usepackage{geometry}
\usepackage{amsmath,amsthm,amsfonts,amssymb,mathtools,mathrsfs,bussproofs,stmaryrd}
\usepackage{biblatex}
\usepackage{hyperref,cleveref,caption}
\usepackage{multirow}
\usepackage{graphicx}
\usepackage[all]{xy}
\usepackage[toc,page]{appendix}

\newtheorem{Proposition}{Proposition}[section]

\newtheorem{lemma}[Proposition]{Lemma}

\newtheorem{fact}[Proposition]{Fact}
\newtheorem{notation}[Proposition]{Notation}
\newtheorem{theorem}[Proposition]{Theorem}
\newtheorem{corollary}[Proposition]{Corollary}
\theoremstyle{definition}

\newtheorem{definition}[Proposition]{Definition}

\newtheorem{remark}[Proposition]{Remark}

\newcommand{\Po}{\mathscr{P}}
\newcommand{\RO}{\mathsf{RO}}

\newcommand{\sset}{\subseteq}
\newcommand{\rset}{\supseteq}
\newcommand{\me}{\wedge}
\newcommand{\bigme}{\bigwedge}
\newcommand{\jo}{\vee}
\newcommand{\bigjo}{\bigvee}

\newcommand{\spa}{\mathcal}
\newcommand{\inv}{^{-1}}
\newcommand{\frec}{\mathfrak{F}}
\newcommand{\NSA}{\,^*\!\mathcal{A}}
\crefname{lemma}{Lemma}{Lemmas}
\crefname{theorem}{Theorem}{Theorems}
\crefname{definition}{Definition}{Definitions}
\crefname{proposition}{Proposition}{Propositions}
\crefname{notation}{Notation}{}
\crefname{corollary}{Corollary}{Corollaries}
\crefname{section}{Section}{Sections}
\crefname{appendix}{Appendix}{Appendices}

\usepackage{verbatim}
\usepackage{tikz}
\usetikzlibrary{positioning}
\usepackage{caption,subcaption}
\newcommand{\cat}{\mathbf}

\newcommand{\la}{\mathscr{L}}
\newcommand{\itp}{\mathscr{I}}
\newcommand{\st}{\,^\dagger\!}
\newcommand{\fr}{\Vdash}
\newcommand{\nfr}{\nVdash}
\newcommand{\fhyl}{^\dagger\!\mathcal{R}}
\newcommand{\fm}{\,^\dagger\!\mathcal{M}}
\newcommand{\RR}{\spa{R}}
\newcommand{\sq}{\}_{i \in \omega}}
\renewcommand{\phi}{\varphi}

\begin{document}

\title{A Semi-Constructive Approach to The Hyperreal Line}
\author{Guillaume Massas}
\date{January 25, 2022}

\maketitle

\begin{abstract}
    Using a recent alternative to Tarskian semantics for first-order logic, known as \textit{possibility semantics}, I introduce an alternative approach to nonstandard analysis that remains within the bounds of \textit{semi-constructive} mathematics, i.e., does not assume any fragment of the Axiom of Choice beyond the Axiom of Dependent Choices. I define the $F$\textit{-hyperreal line} $\fhyl$ as a possibility structure and show that it shares many fundamental properties of the classical hyperreal line, such as a Transfer Principle and a Saturation Principle. I discuss the technical advantages of $\fhyl$ over some other alternative approaches to nonstandard analysis and argue that it is well-suited to address some of the philosophical and methodological concerns that have been raised against the application of nonstandard methods to ordinary mathematics.
\end{abstract}


\begin{ams}
00A30, 03E25, 03E40, 03H05
\end{ams}

\tableofcontents

\section{Introduction} \label{intro}

 Non-standard analysis is a branch of mathematical logic which focuses on the application of powerful metamathematical methods to ordinary mathematics, following the groundbreaking work of Abraham Robinson \cite{robinson1996non} in the 60s. Although it has now developed into a diverse field interested in more than one structure, it originated from the development of a system of hyperreal numbers which is powerful and suggestive enough to serve as an alternative foundation to classical analysis. The common way of introducing such numbers is via the hyperreal line $^*\RR$, constructed as some ultrapower of the standard real line modulo a non-principal ultrafilter on the set $\omega$ of the natural numbers \cite{keisler1994hyperreal}. The main advantage of this ultrapower construction is that it can easily be seen to satisfy two fundamental principles: the \textit{Transfer Principle}, which guarantees that the hyperreal line is \textit{standard enough} to have the same first-order theory as the real line, and the (Countable) \textit{Saturation Principle}, which allows one to derive the existence in $^*\RR$ of many \textit{non-standard objects} that can play a key role in simplifying many classical arguments in analysis.\\

The existence of non-principal ultrafilters on $\omega$, however, is a set-theoretic result that exceeds the resources of semi- or quasi-constructive mathematics as described in \cite{schechter1996handbook}, i.e., classical mathematics that can be carried out in $ZF$ with the addition of the Axiom of Dependent Choices ($DC$). As Schechter writes (\cite[Chap.~14]{schechter1996handbook}), $ZF+DC$ is a natural setting for analysis, as it neither assumes nor rejects some of the more paradoxical consequences of the Axiom of Choice in analysis, such as the existence of a well-ordering of the reals or of non-Lebesgue measurable sets. Free ultrafilters on infinite Boolean algebras exist in the presence of the Ultrafilter Lemma, a fragment of the Axiom of Choice which is known to be independent from $ZF+DC$ \cite{Feferman65}. As one can  straightforwardly define non-Lebesgue measurable sets from a non-principal ultrafilter on $\omega$, such objects go beyond the resources of semi-constructive mathematics, and they form what Schechter calls \textit{intangibles}, objects whose existence can be proved in classical mathematics even though they cannot themselves be \textit{explicitly} constructed.

On the other hand, in a strictly semi-constructive setting, preserving both the Transfer and Saturation Principles seems hopeless. A straightforward argument shows that the existence of an elementary embedding $e: \RR \to \mathcal{M}$ in a sufficiently rich language $\la$ implies the existence of a non-principal ultrafilter on $\omega$. Indeed, the Countable Saturation Principle implies the existence of at least one infinite hypernatural integer $N$, larger than any finite number $n$. If $\la$ contains a predicate symbol for every set of natural numbers, then one may consider the set $\mathcal{U} = \{A \sset \omega : \mathcal{M} \models A(N)\}$. By the Transfer Principle and the truth-conditions of Tarskian semantics, $\mathcal{U}$ is then easily seen to be a non-principal ultrafilter on $\omega$. We are therefore faced with the following trilemma: we must either give up the Transfer Principle, the Saturation Principle, or Tarskian semantics. It is this third option that I would like to pursue here, by appealing to an alternative semantics for classical logic known as possibility semantics. As we will see, possibility semantics allows for the construction of a first-order structure that shares many features with the classical hyperreal line, including versions of the Transfer and Saturation Principles, yet is entirely independent of the Prime Ideal Theorem. The key idea is to replace one single, static first-order model determined by an ultrafilter with a system of \textit{viewpoints}, partial approximations of such a model each determined by a filter. This idea itself is not new, although, as I argue below, it can be presented in a particularly powerful, simple and concrete way thanks to the machinery of possibility semantics. Not only does the resulting structure have some technical advantages over some other alternative approaches to nonstandard analysis, it also has some methodological and conceptual advantages over the classical hyperreal line obtained via ultrapowers. In particular, I argue that it is well-suited to address concerns that have been raised in the literature regarding the application of nonstandard methods to standard mathematics, including issues about the purity  (in the sense of \cite{arana2017alleged,arana2011purity}) of nonstandard methods and the canonicity of the hyperreal line.\\

The paper is organized as follows. In the first section, I recall the main features of possibility semantics for first-order logic, before introducing a new kind of construction, \textit{Fréchet powers}, which constitue an alternative to classical ultrapowers modulo a non-principal ultrafilter on $\omega$. This allows me to define the $F$-hyperreal line $\fhyl$ as a Fréchet power of the reals and to explore some of its mathematical properties. In particular, I show that versions of the Transfer and Saturation Principles of classical nonstandard analysis hold on $\fhyl$ and that the $F$-hyperreals share many features with the classical hyperreal line, including a natural characterization of continuous real-valued functions and a robust theory of internal sets. 

\cref{asym,dyn,gen} are then devoted to assessing the technical merits of $\fhyl$ over other alternatives to classical nonstandard analysis. As we will see, the central idea behind the $F$-hyperreal line, i.e., to avoid relying on a single non-principal ultrafilter by working with a system of filters instead, can be traced back to several distinct mathematical endeavors. I distinguish three kinds of alternative approaches to nonstandard analysis and argue that $\fhyl$ constitutes a natural meeting point for all three approaches. Indeed, it can be seen as a suitable framework to develop the natural and historically influential idea that the properties of infinite sequences ``in the limit'' are determined by the properties of all but finitely many of their values, thus continuing some historically-minded work by Laugwitz and offering a formal counterpart to a more recent proposal of Tao (\cref{asym}). On the other hand, the dynamic aspect of possibility semantics, and its connections with sheaf semantics and topos theory, also allows one to think of $\fhyl$ as a ``varying'' reduced power of the reals (\cref{dyn}). Finally, $\fhyl$ can be analysed as a Boolean-valued model of analysis in which viewpoints are partial approximations of classical ultrapowers, which then arise as generic constructions in the precise sense of forcing (\cref{gen}).

Finally, in \cref{obj}, I turn to some philosophical objections to the application of nonstandard methods in ordinary mathematics and argue that the $F$-hyperreal line offers a novel way to tackle the problems raised by these objections. In particular, the fact that $\fhyl$ can be constructed in a semi-constructive setting means that it is better suited to address purity concerns regarding the use of nonstandard methods in analysis, and the use of a system of filters in lieu of a single ultrafilter turns it into a canonical structure.

\section{Constructing the F-Hyperreal Line}\label{PossPow}

In this section, I introduce the F-hyperreal line as a concrete possibility structure. I start by recalling some basic features of possibility semantics for first-order logic, before introducing \textit{Fr\'echet powers} as a semi-constructive analogue of ultrapowers, and defining the $F$-hyperreal line $\fhyl$ as a Fréchet power of the real line. 

\subsection{Possibility Semantics}\label{poss}

Possibility semantics originated as an alternative semantics for classical propositional and modal logic \cite{Hol16,Humberstone}. From a mathematical perspective, it shares a lot with the technique of forcing in set theory, especially the correspondence between forcing extensions and Boolean-valued models \cite{bell1977boolean}, and with a specific kind of sheaf semantics in topos theory \cite{maclane2012sheaves}. These connections will be discussed in more detail in \cref{dyn,gen} below. From a philosophical perspective, possibility semantics can be seen as a generalization of Kripke semantics for propositional modal logic. In Kripke semantics, models are collections of possible worlds, which can be identified with maximal sets of consistent formulas. By contrast, the points in a possibility structure are ``partial'' worlds, i.e., they correspond to consistent sets of formulas that may contain neither a formula nor its negation. Points in a possibility structure are naturally ordered by how informative they are, where a point is more informative than another if it satisfies more formulas. In that respect, possibility semantics can be seen as a variation on Kripke semantics for propositional intuitionistic logic. The crucial difference however lies in the semantic clauses for disjunction: in possibility semantics, a disjunction may be satisfied at a point $p$ without any of the disjuncts being satisfied at $p$. This is what allows the law of excluded middle to always be true at a partial world, even when that world is not maximally determined.

Because the points in a possibility structure are partial possibilities, instead of maximal possible worlds, the construction of possibility structures is typically more constructive than the construction of standard, possible worlds models. From the mathematical side, this allows for a choice-free representation theorem for Boolean algebras, which is also lifted to a full categorical duality in \cite{hol19}. This is in stark contrast with the standard Stone duality, in which the Axiom of Choice, and more precisely the Ultrafilter Lemma (any filter on a Boolean algebra extends to an ultrafilter) plays a crucial role. In possibility semantics, the role usually played by maximal, non-constructive objects (possible worlds, ultrafilters), is taken up by partial, constructive objects (partial worlds, filters). As we will see below, it is precisely this feature that will allow for a construction of the hyperreals that bypasses the usual reliance on a non-principal ultrafilter on $\omega$.

Possibility semantics has also been developed for first-order languages, although the field is currently less advanced. The seminal work is a manuscript by van Benthem \cite{vanBenthem81}, and a detailed presentation can be found in \cite{holposs}. A very attractive feature of first-order possibility semantics is the existence of a fully constructive completeness theorem for first-order logic, while the usual completeness theorem for first-order logic with respect to Tarskian semantics is known to require the Axiom of Choice (see \cite{bell}, p.~140). In what follows, I define first-order possibility structures in a slightly different fashion than in \cite{holposs}, by introducing a slight alteration to the way Holliday defines the interpretation of functions symbols. As Holliday remarks in \cite[fn.~20]{holposs}, this modification is however of no real significance for the constructive completeness theorem obtained in \cite[Theorem 4.3.8]{holposs}.

\begin{definition}[Possibility structure] \label{def1}
Let $\la$ be a first-order language. A possibility structure is a tuple $(\mathfrak{P}, D, \itp)$ such that:
\begin{itemize}
    \item $\mathfrak{P} = (P, \leq)$ is a partially ordered set (\textit{poset} for short) of \textit{viewpoints};
    \item $D$ is a set of \textit{guises};
    \item $\itp$ is a function mapping any $p \in P$ and any $n$-ary relation symbol $R$ (including the equality symbol $=$) in $\la$ to a subset of $D^n$, and any $n$-ary function symbol $f$ in $\la$ to a function from $D^n$ into $D$, so that for any $p, q \in P$, any $n$-ary relation $R$ in $\la$, any $n$-ary function symbol $f$ and any $n$-tuple $\overline{a}$, the following conditions hold:
    \begin{description}
        \item[Persistence] If $\overline{a} \in \itp(p,R)$ and $q\leq p$, then $\overline{a} \in \itp(q,R)$;
        \item[Refinability] If $\overline{a} \notin \itp(p,R)$, then there is $q \leq p$ such that for all $r \leq q$, $\overline{a} \notin (r,R)$;
        \item[Equality-as-equivalence] $\itp(p,=)$ is an equivalence relation on $D\times D$;
        \item[Equality-as-congruence] if $\overline{a} = (a_1,...,a_n)$ and $\overline{b} = (b_1,...,b_n)$ is an $n$-tuple such that $(a_i,b_i) \in \itp(p,=)$ for all $i \leq n$, then $(\itp(f)(\overline{a}),\itp(f)(\overline{b})) \in \itp(p,=)$, and $\overline{a} \in \itp(p,R)$ iff $\overline{b} \in \itp(p,R)$.
    \end{description}
\end{itemize}
\end{definition}

Intuitively, any point in a possibility structure provides us with a ``partial viewpoint'' on how the model actually looks. In particular, the ``guises'' in the domain $D$ are not objects themselves, but mere distinct ways of presenting objects. Two different guises may actually correspond to one and the same object from one viewpoint, hence the need for the relation symbol for equality to be interpreted as an equivalence relation, rather than as strict equality. As we move from less informative viewpoints to more informative ones, more guises are identified with one another, and more information is gained regarding the relations that hold between the objects that the guises designate. The persistence condition encapsulates the idea that the partial order on the poset does indeed capture the increase of information between viewpoints: no information is lost when we move from a viewpoint to a stronger viewpoint.\footnote{For a reader unfamiliar with forcing, it might seem counterintuitive that a ``smaller'' viewpoint is also a ``stronger'' one. But if one take a viewpoint $p$ to stand in for ``all the ways compatible with $p$ in which the model could actually be'' and the relation $\leq$ to indicate containment, then it is straightforward to see that $p\leq q$ precisely when ``all the ways compatible with $p$ in which the model could actually be'' are also "ways compatible with $q$ in which the model could actually be'', thus meaning that $p$ imposes stronger conditions than $q$ on what the model could actually be. It is worth mentioning that nothing really hinges upon this choice of defining ``less than'' as ``stronger than'', rather than ``weaker than'', although, as it will become more apparent later on, doing so underscores the tight connection between possibility semantics, forcing, and sheaf semantics.} The refinability condition, by contrast, ensures that our information states, while partial, are as informative as they could be. More precisely, it is the contrapositive of the following ``sure thing'' principle: if no further refinement of our current viewpoint could make sure that the tuple $\overline{a}$ does not stand in relation $R$, then we might as well conclude already that the tuple $\overline{a}$ does stand in relation $R$. This refinability condition is what sets apart possibility semantics from Kripke semantics for first-order intuitionistic logic. It also appears in the inductive definition of satisfaction of a formula, to which I now turn.

\begin{definition}[Forcing relation]\label{semdef}
Let $\mathfrak{M} = (\mathfrak{P},D,\itp)$ be a $\la$-possibility structure for a first-order language $\la$. The forcing relation $\fr$ is inductively defined for any $p \in P$, any $\la$-formula $\phi(\overline{x})$ with $n$ free variables, and any $n$-tuple $\overline{a}$ of elements of $D$ as follows:
\begin{itemize}
    \item If $\phi:=R(t_1(\overline{x_1}),...,t_n(\overline{x_j}))$, where $t_1,...,t_n$ are $\la$-terms of arity $n_1,...,n_j$ summing up to $n$ and $R$ is a $j$-ary relation symbol, then $$p \fr R(t_1(\overline{a_1}),...,t(\overline{a_j})) \text{ iff } (\itp(t_1)(\overline{a_1}),...,\itp(t_n)(\overline{a_n})) \in \itp(p,R),$$ where $\overline{a} = \overline{a_1}...\overline{a_j}$, $\overline{a_i}$ is an $n_i$-ary tuple for any $i \leq j$, and the interpretation of an $\la$-term is inductively defined from the interpretation of function symbols as usual;
    \item If $\phi := \neg \psi$, then $p \fr \phi$ iff for all $q \leq p$, $q \nfr \psi$;
    \item If $\phi := \psi \me \chi$, then $p \fr \phi$ iff $p \fr \psi$ and $ p \fr \chi$;
    \item If $\phi := \psi \jo \chi$, then $p \fr \phi$ iff for all $q \leq p$ there is $r \leq q$ such that $r \fr \psi$ or $r \fr \chi$;
    \item If $\phi := \psi \to \chi$, then $p \fr \phi$ iff for all $q \leq p$, $q \fr \psi$ implies $q \fr \chi$;
    \item If $\phi := \forall x \psi$, then $p \fr \phi$ iff $p \fr \psi(a)$ for any $a \in D$;
    \item If $\phi := \exists x \psi$, then $p \fr \phi$ iff for all $q\leq p$ there is $r \leq q$ such that $r \fr \phi(a)$ for some $a \in D$.
\end{itemize}

\noindent Given a $\la$-formula $\phi$ and an $n$-tuple $\overline{a}$, $\phi(\overline{a})$ is \textit{valid} in $\mathfrak{M}$ (denoted $\mathfrak{M} \models \phi(\overline{a})$) if $p \fr \phi(\overline{a})$ for all $p \in P$.

\end{definition}

The forcing clauses introduced above are of course reminiscent of both Kripke semantics and the forcing relation in set theory. The refinability condition appears in the clauses for disjunctions and existentials, which can be straightforwardly derived from the clauses for negations, conjunctions and universals and De Morgan's laws. From an algebraic perspective, they also ensure that every $\la$-sentence is given a value in a Boolean algebra of subsets of the poset $P$. More precisely, we may introduce the following notation:

\begin{notation}\label{not1} Given an $\la$-formula $\phi(\overline{x})$ with $n$ free variables and a $n$-tuple $\overline{a}$, let $\llbracket \phi(\overline{a})\rrbracket = \{p \in P: p \fr \phi(\overline{a})\}$. 
\end{notation} 

Then a simple induction on the complexity of formulas shows that $\llbracket \phi(\overline{a})\rrbracket$ is always a \textit{regular-open} subset of $P$, i.e., that for any $p \in P$:
\begin{align*}
    p \in \llbracket\phi(\overline{a})\rrbracket \Leftrightarrow \forall q \leq p \, \exists r \leq q: r \in \llbracket \phi(\overline{a}) \rrbracket.
\end{align*}
Since the regular open sets of any poset always form a Boolean algebra \cite{hol19,givant2008}, one may therefore think of a possibility structure as a generalized Tarskian model in which sentences are Boolean-valued rather than $2$-valued (the relationship with Boolean-valued models will be discussed in more detail in \cref{gen}). In fact, Tarskian structures are precisely those possibility structures $\mathfrak{M} = (\mathfrak{P},D,\itp)$ in which $\mathfrak{P}$ is a single element poset, and the equality symbol is interpreted as the identity relation. As a consequence, first-order logic is still sound with respect to this larger class of models, but the proof that it is also complete can now be carried out even in the absence of the Axiom of Choice:

\begin{theorem}[\cite{holposs}, ZF]
For any first-order language $\la$, first-order logic is sound and complete with respect to the class of $\la$-possibility structures.
\end{theorem}
\subsection{Fréchet Powers} \label{frec}

Having introduced the basics of possibility semantics for first-order logic, we are now in a position to define the main technical construction of this paper, Fréchet powers. In the setting of possibility semantics, Fréchet powers are an alternative to the ultrapowers modulo a non-principal ultrafilter on $\omega$ of Tarskian semantics, which I will call \textit{Luxemburg ultrapowers} from now on, as, according to Keisler in \cite{keisler1994hyperreal}, their relevance to  nonstandard analysis was first identified by Luxemburg in \cite{luxemburg1966non}. An analogue of classical ultraproducts in the setting of possibility semantics was already discussed in \cite{vanBenthem81}. However, van Benthem is interested in defining a general notion of product for all possibility structures, while I will only be interested in defining Fréchet powers of Tarskian models, which simplifies van Benthem's construction quite significantly.  It is also worth mentioning that Fréchet powers can be defined constructively, and that only the Axiom of Countable Choices, which is a theorem of $ZF+DC$, will be required to prove the Transfer Principle in \cref{princ}.

Let us start with the definition of the poset that will play a crucial role throughout this paper. Recall first that a \textit{filter} on $\omega$ is a subset $F$ of $\Po(\omega)$ which is closed under intersections and supersets, meaning that for any $A,B \sset \omega$, $A, B \in F$ implies $A \cap B \in F$, and $A \sset B$ and $A \in F$ together imply that $B \in F$. A filter $F$ on $\omega$ is \textit{non-principal} if it contains no finite subsets of $\omega$.

\begin{definition}[Fréchet Poset]
The \textit{Fréchet poset} is the poset $\frec = (\spa{F}, \rset)$, where $\spa{F}$ is the set of all non-principal filters on $\omega$, ordered by reverse inclusion.

\end{definition}

The Fréchet poset takes its name from the fact that any element in $\frec$ extends the filter of all cofinite subsets of $\omega$, which is usually called the \textit{Fréchet filter}, and that I will denote by $F_0$.

\begin{definition}[Fréchet Power]
Let $\spa{M}$ be a Tarskian model of a first-order language $\la$. The \textit{Fréchet power} $\fm = (\mathfrak{F}, M^\omega, \itp)$ of $\spa{M}$ is the $\la$-possibility structure determined by the following data:
\begin{itemize}
    \item $(\mathfrak{F}, \rset)$ is the Fréchet poset;
    \item $M^\omega$ is the set of all functions from $\omega$ into the domain of $\spa{M}$;
    \item For any $n$-ary relation symbol $R \in \la$ and any $n$-tuple $\overline{a}$, $\overline{a} \in \itp(F,R)$ iff $\{i \in \omega: \spa{M} \models R(\overline{a}(i))\} \in F$; 
    \item For any $n$-ary function symbol $f(\overline{x}) \in \la$ and any $a_1,...,a_n \in M^\omega$, $\itp(f)(a_1,...,a_n) $ is the function $g : \omega \to M$ such that for any $i \in \omega$, $g(i) = f(a_1(i),...a_n(i))$. 
    \end{itemize}
\end{definition}

Intuitively, the Fréchet power of $\spa{M}$ can be thought of as a collection of partial approximations of what a classical ultrapower of $\spa{M}$ modulo a non-principal ultrafilter $\spa{U}$ on $\omega$ might look like. Elements in such an ultrapower are equivalence classes of functions in $M^\omega$, where two functions $f$ and $g$ are considered equivalent if they agree on a $\spa{U}$-large set, meaning that $\{i \in \omega : f(i) = g(i)\} \in U$. By contrast, in $\spa{M}_\frec$, a function $f$ is a mere guise for what its equivalence class would be in a Tarskian ultrapower, and consequently, $f$ is identified at some $F \in \frec$ with another function $g$ precisely when $f$ and $g$ agree on a large enough set from the viewpoint of $F$, i.e., whenever $\{i \in \omega : f(i) = g(i)\} \in F$. This identification of viewpoints in $\frec$ with partial approximations of an ultrapower will be made more precise in \cref{ult}. For now, we must still verify that $\spa{M}_\frec$ is a possibility structure according to \cref{def1}. In order to do this, we introduce the following notation, which is standard in the literature on ultrapowers.
\begin{notation}\label{not2}
Let $\spa{M}$ be a Tarskian $\la$-structure. Given a $n$-tuple $\overline{f} = (f_1,...f_n)$ of functions in $M^\omega$, let $\overline{f(i)}$ be the $n$-tuple of elements $(f_1(i),...,f_n(i)) \in M^n$  for any $i \in \omega$. Given any $\phi(\overline{x})$ an $\la$-formula in $n$-variables, and any $n$-tuple $\overline{f}$, let $||\phi(\overline{f})|| = \{i \in \omega : \spa{M} \models \phi(\overline{f(i)})\}$. 
\end{notation}

Given an $\la$-sentence $\phi$, one should distinguish $||\phi||$, which is a subset of $\omega$, from $\llbracket \phi \rrbracket$, introduced in \cref{not1}, which is a subset of $\frec$. As we will see in \cref{los}, there is however a tight connection between the two sets. Using the notation just introduced, the interpretation function $\itp$ in $\spa{M}_\frec$ can be conveniently rephrased. Given a viewpoint $F \in \frec$, an $n$-ary relation symbol $R$ and an $n$-tuple $\overline{a}$ in $M^\omega$, $\overline{a} \in \itp(F,R)$ iff $||R(\overline{a})|| \in F$. 

\begin{lemma} \label{lma1}
The Fréchet power of any $\la$-structure $\spa{M}$ is an $\la$-possibility structure.
\end{lemma}

\begin{proof}
Fix filters $F, G \in \frec$, an $n$-tuple $\overline{a}$ of elements in $M^\omega$, an $n$-ary relation symbol $R$ and an $n$-ary function symbol $f$. We check the four conditions in \cref{def1} in turn.
\begin{description}
\item[Persistence] Suppose $G \rset F$ and $\overline{a} \in \itp(F,R)$. Then $||R(\overline{a})|| \in F$, and therefore $||R(\overline{a})|| \in G$, which implies that $\overline{a} \in \itp(G,R)$.
\item[Refinability] Assume $\overline{a} \notin \itp(F,R)$. Then $||R(\overline{a})|| \notin F$, which, by basic properties of filters, implies that there is a filter $G \rset F \cup (\omega \setminus ||R(\overline{a})||)$. But then clearly $||R(\overline{a})|| \notin H$ for any $H \rset G$, and therefore $\overline{a} \notin \itp(H,R)$ for any $H \rset G$.
\item[Equality-as-equivalence] Note first that for any $a \in M^\omega$, $||a = a|| = \omega$. Moreover, for any $a,b,c \in M^\omega$, $||a = b|| \sset ||b = a||$, and $||a = b|| \cap ||b = c|| \sset ||a = c||$. This shows that for any filter $F$, $||a = a|| \ \in F$, $||a = b|| \in F$ implies $||b = a|| \in F$, and $||a = b|| \text{ and } ||b = c|| \in F$ together imply $||a = c|| \in F$. Thus $\itp(F,=)$ is an equivalence relation on $M^\omega$ for any $F \in \frec$.
\item[Equality as congruence] Suppose that $(a_i,b_i) \in \itp(F,=)$ for any $i \leq n$, and let $\overline{a} = (a_1,...,a_n)$ and $\overline{b} = (b_1,...,b_n)$. Note that $\bigcap_{i \leq n} ||a_i = b_i|| \sset ||f(\overline{a}) = f(\overline{b})||$, hence that $||f(\overline{a}) = f(\overline{b})|| \in F$ and thus that $(f(\overline{a},f(\overline{b}))\in \itp(F,R)$. Moreover, assume that $\overline{a} \in \itp(F,R)$. Note that $\bigcap_{i \leq n} ||a_i = b_i|| \cap ||R(\overline{a})|| \sset ||R(\overline{b})||$, which implies that $||R(\overline{b})|| \in F$ and hence that $\overline{b} \in \itp(F,R)$. \qedhere
\end{description}
 \end{proof}
 
 The previous lemma therefore justifies us in introducing the main construction of this paper, the F-hyperreal line, as the Fréchet power of the reals in a significantly rich first-order language.

\begin{definition}[F-Hyperreal Line] \label{def2}
Let $\la$ be a first-order language with a relation symbol for every finitary relation on $\mathbb{R}$ and a function symbol for every finitary function on $\mathbb{R}$, and let $\RR$ be the Tarskian $\la$-structure with domain $\mathbb{R}$ in which every relation or function symbol is interpreted as the relation or function on $\RR$ it corresponds to.  The \textit{F-hyperreal line} is the Fréchet power of $\RR$, i.e., the $\la$-possibility structure $\fhyl = (\frec,\RR^\omega,\itp)$, where for any $n$-tuple $\overline{a} \in \RR^\omega$, $\itp(f)(\overline{a})(i) = f(\overline{a}(i))$ for any $n$-ary function symbol $f \in \la$ and any $i \in \omega$, and $\overline{a} \in \itp(F,R)$ iff $||R(\overline{a})|| \in F$ for any $n$-ary relation symbol $R$ in $\la$ and any $F \in \frec$.
\end{definition}

\section{Some Basic Properties of $\fhyl$} \label{fhyl}

Having reached a definition of the $F$-hyperreal line, we may now investigate the properties of this structure more closely. As my main interest is to highlight its similarity with the classical hyperreal line as presented for example in \cite{goldblatt2012lectures}, I will be focusing on some of the most well-known features of Robinsonian nonstandard analysis, such as the definition of continuity and the theory of internal sets. I show first that a version of the Transfer Principle holds on $\fhyl$, which allows me to establish that $\fhyl$ contains some infinitesimals and to give a natural characterization of continuity which mirrors the classical nonstandard one. I will then show that a notion of internal set can be fruitfully developed on $\fhyl$ for which several properties of classical internal sets hold, including a version of the Countable Saturation Principle. This establishes that possibility semantics offers a way out of the trilemma presentend in the introduction: we can preserve the Transfer and Saturation Principles in a semi-constructive setting, if we are willing to move away from Tarskian semantics.

\subsection{Transfer Principle and Continuity on $\fhyl$} \label{princ}

As we shall see in the next two sections, the F-hyperreal line is sufficiently well-behaved to share many features with the classical hyperreal line. In particular, natural versions of two of the most fundamental tools of nonstandard analysis, the Transfer and Saturation Principles, hold in $\fhyl$. That classical ultrapowers of the reals satisfy the Transfer Principle is a direct consequence of \L\'os's Theorem \cite{bell}. As we shall see, \L\'os's Theorem generalizes in a natural way to Fréchet powers and enables us to understand the forcing relation in $\fhyl$ very concretely.

\begin{theorem}[\L\'os's Theorem for $\fhyl$] \label{los}
For any $\la$-formula $\phi(\overline{x})$ with $n$ free variables, any $n$-tuple $\overline{a} \in \fhyl$ and any $F \in \frec$, $F \fr \phi(\overline{a})$ iff $||\phi(\overline{a})|| \in F$. 
\end{theorem}

\begin{proof}
The proof proceeds by induction on the complexity of formulas. The atomic case follows immediately from the definition of the interpretation function $\itp$. For the inductive case, I only treat the cases of negation, conjunction and existential quantification, since the other Boolean connectives and quantifiers are definable from this set.
\begin{itemize}
    \item Suppose $\phi:= \neg \psi$. Then $F \fr \phi(\overline{a})$ iff for all $G \rset F \in \frec$, $G \nfr \psi(\overline{a})$. By induction hypothesis, the latter is equivalent to the claim that for all $G \rset F \in \frec$, $||\psi(\overline{a})|| \notin G$. By an argument similar to \cref{lma1}, this in turn is equivalent to requiring that $\omega \setminus ||\psi(\overline{a})|| \in F$. Clearly, $\omega \setminus ||\psi(\overline{a})|| = ||\neg\psi(\overline{a})||$, from which it follows that $F \fr \phi(\overline{a})$ iff $||\phi(\overline{a})|| \in F$.
    \item Suppose $\phi:= \psi_1 \me \psi_2$. Then $F \fr \phi(\overline{a})$ iff $F \fr \psi_1(\overline{a})$ and $F \fr \psi_2(\overline{a})$. By induction hypothesis, this is equivalent to $||\psi_1(\overline{a})|| \in F$ and $||\psi_2(\overline{a})|| \in F$, which in turn is equivalent to $||\psi_1(\overline{a})|| \cap ||\psi_2(\overline{a})|| \in F$ as $F$ is a filter. But the latter is clearly equivalent to $||\phi(\overline{a})|| \in F$.
    \item Suppose $\phi:= \exists x \psi(x)$. If $F \fr \exists x \psi(x, \overline{a})$, then for any $G \rset F$, there is $b \in \fhyl$ and $H \rset G$ such that $H \fr \psi(b, \overline{a})$. By induction hypothesis, this means that for all $G \rset F$, there is $H \rset G$ and $b \in \fhyl$ such that $||\psi(b,\overline{a})|| \in H$. Since, for any $b \in \fhyl$, $||\psi(b, \overline{a}))|| \sset ||\exists x \psi(x, \overline{a})||$, it follows that for any $G \rset F$, there is $H \rset G$ such that $||\exists x \psi(x, \overline{a})|| \in H$. This in turn implies that for any $G \rset F$, $\omega \setminus ||\exists x \psi(x, \overline{a})|| \notin G$, from which it follows that $||\exists x \psi(x, \overline{a})|| \in F$. \\
    Conversely, suppose $||\exists x \psi(x, \overline{a})|| \in F$. Define $b : \omega \to \RR$ such that $\spa{R} \models \psi(b(i),\overline{a}(i))$ whenever $i \in ||\exists x \psi(x, \overline{a})||$, and $b(i)$ is arbitrary otherwise. Then $||\psi(b, \overline{a})|| \rset ||\exists x \psi(x, \overline{a})||$, hence $||\psi(b,\overline{a})|| \in F$. By induction hypothesis, this means that $F \fr \phi(b,\overline{a})$ and hence $F \fr \exists x \psi(x, \overline{a})$, which completes the proof. \qedhere
\end{itemize}
\end{proof}
A few remarks can be made regarding the proof of \cref{los}:

\begin{remark} \label{rmk1}

\begin{itemize}
\item[]
    \item Using the notation in \cref{not1,not2}, \cref{los} can be succinctly phrased as follows: for any $\la$-formula $\phi(\overline{x})$ with $n$ free variables and any $n$-tuple $\overline{a} \in \fhyl$, \[\llbracket \phi(\overline{a}) \rrbracket = \{F \in \frec : ||\phi(\overline{a})|| \in F\}.\]
    \item Closer inspection of the existential case of the proof reveals that $\fhyl$ is a \textit{full} model in the following sense: for any formula $\phi(x, \overline{y})$, any tuple $\overline{a}$ and any $F \in \frec$, $F \fr \exists x \phi(x, \overline{a})$ iff there is $b \in \fhyl$ such that $F \fr \phi(b, \overline{a})$. In other words, although standard possibility semantics allows for the witness to an existential sentence to vary across refinements, a stronger condition actually holds in the case of Fréchet powers.
    \item The right-to-left direction of the existential case requires the Axiom of Countable Choices in order for the witness $b$ to be defined. This is the only part of the theorem that requires some fragment of the Axiom of Choice.
    \item The proof does not use at any point the fact that all filters are extensions of the Fréchet filter, or filters on $\omega$. In fact, the same proof would go through if $\frec$ were to be replaced with any collection of filters on a countable set $I$ that is closed under supersets.
    \item Similarly, the proof would go through for any expansion of the language $\la$, provided of course that the interpretation of the new relation and function symbols is given by the conditions of the theorem.
\end{itemize}
\end{remark}

The Transfer Principle for $\fhyl$ can then be obtained as an easy corollary.

\begin{corollary}[Transfer Principle] \label{trans}
There exists a function $\delta: \RR \to \fhyl$ such that for any $\la$-sentence $\phi(\overline{x})$ and any tuple of reals $\overline{r}$, $\RR \models \phi(\overline{r})$ iff the Fréchet filter $F_0 \fr \phi(\overline{\delta(r)})$ iff $\fhyl \models \phi(\overline{\delta(r)})$. In particular, for any $\la$-sentence $\phi$, $\RR \models \phi$ iff $ \fhyl \models \phi$ iff $F_0 \fr \phi$.
\end{corollary}

\begin{proof}
For any $r \in \RR$, let $\delta(r)(i) = r$ for all $i \in \omega$. Then for any $\la$-sentence $\phi(\overline{x})$, any tuple of reals $\overline{r}$ and any $i \in \omega$, $\RR \models \phi(\overline{r})$ iff $ i \in ||\phi(\overline{\delta(r)})||$. In particular, this means that $||\phi(\overline{\delta(r)})||$ is either $\emptyset$ or $\omega$, meaning that $F \fr \phi(\overline{\delta(r)})$ iff $||\phi(\overline{\delta(r)})|| = \omega$ iff $F_0 \fr \phi(\overline{\delta(r)})$. 
\end{proof}

\L\'os's Theorem and the Transfer Principle for $\fhyl$ allows us to draw two important consequences. First, there is a rather robust sense in which $\RR$ elementarily embeds into $\fhyl$: even though $\fhyl$ is not a classical Tarskian structure, if one thinks of every viewpoint in $\frec$ as a partial approximation of what such a Tarskian structure could look like, then, by the Transfer Principle, any such approximation makes sure that this Tarskian structure would have the same first-order theory as $\RR$ and would also contain a copy of the reals. By contrast, if one prefers to think of every point in $\frec$ as a ``local'' viewpoint on a non-constant, ever-changing hyperreal line, then the Transfer Principle guarantees that the reals are fixed, unmovable points of this line. The second consequence that one can draw is that validity on $\fhyl$ coincides with satisfaction at $F_0$: a formula will be satisfied at any $F \in \frec$ precisely if it is forced by $F_0$. Whether one treats points in $\frec$ as approximations of a Tarskian model, or as snapshots of a changing model, this guarantees that the formulas that are true ``for sure'' or ``always'', are precisely those formula $\phi(\overline{a})$ that are true in $\RR$ ``cofinitely often'', i.e., all formulas such that $||\phi(\overline{a})||$ is a cofinite subset of $\omega$. $F_0$ can therefore be thought of as a ``generic'' viewpoint on the F-hyperreal line: it forces precisely what must be true of it and nothing else. I will come back to the significance of this fact in \cref{obj} below.\\

Just as in Luxemburg ultrapowers, we may now define the nonstandard extension of an $n$-ary relation $R$ on $\mathbb{R}$ as an $n$-ary relation $\st R$ on $\fhyl$:

\begin{definition}[Nonstandard Extension]
Let $R$ be an $n$-ary relation on $\mathbb{R}$. The $F$\textit{-nonstandard extension of} $R$, noted $\st R_F$, is the set $\{\overline{a} \in (\fhyl)^n : F \fr R(\overline{a})\} = \itp(F,R)$. Similarly, if $f$ is an $n$-ary function on $\mathbb{R}$, we let $\st f_F$ be the $n$-ary function $\itp(f)$.
\end{definition}

Note that, because of our choice of defining functions on $\fhyl$, the nonstandard extension of a function $f$ on $\mathbb{R}$ is a \textit{bona fide} function that is not relativized to a filter $F$. Nonetheless, whether some $b \in \fhyl$ is the image of a tuple $\overline{a}$ under $\st f$ is  relative to a filter $F$, since equality is interpreted as an equivalence relation on each filter. As is customary in nonstandard analysis, I will identify a function $f$ with its nonstandard extension $\st f$ whenever no ambiguity arises. Let us also introduce the following definitions, which mirror usual notions in nonstandard models:

\begin{definition}\label{halo}
For any $a \in \fhyl$ and any $F \in \frec$, let $a_F = \{b \in \fhyl : F \fr a = b\}$, and let the $F$-\textit{halo} of $a$ be the set $(a)_F = \{ b \in \fhyl : F \fr |a-b| < \frac{1}{n} \text{ for all } n \in \mathbb{N}\}$.
\end{definition}

In other words, for any $a \in \fhyl$ and any $F \in \frec$, $a_F$ is the collection of all guises that are identified from the viewpoint $F$ with the same hyperreal as the one $a$ designates, while $(a)_F$ is the set of all guises which are infinitesimally close to $a$ from the viewpoint $F$. Obviously, $a_F \sset (a)_F$, but the converse is never true:

\begin{lemma}
For any $a \in \fhyl$ and any $F \in \frec$, there is $b \in (a)_F \setminus a_F$.
\end{lemma}

\begin{proof}
Fix $a \in \fhyl$ and $F \in F$, and let $b : \omega \to \mathbb{R}$ be defined as $b(i) = a(i) + \frac{1}{i}$ for all $i \in \omega$. Then note that $||a = b|| = \emptyset$, which means that $F \fr a \neq b$ and hence that $b \notin a_F$. However, $|| |a-b| < \frac{1}{n}|| = \{i \in \omega : i > n\}$ for any $n \in \mathbb{N}$, which means that $|| |a-b| < \frac{1}{n}|| \in F$, since $F$ is non-principal. Hence, by \cref{los}, $b \in (a)_F$.
\end{proof}

This lemma guarantees that we can meaningfully talk about infinitesimals on the $F$-hyperreal line. Of course, whether two elements $a, b \in \fhyl$ are infinitesimally close to one another may vary with $F$. Nonetheless, this still allows us to give a straightforward adaptation of the usual definition of continuity on the hyperreal line.

\begin{definition}[$F$-continuity]
A function $f : \fhyl \to \fhyl$ is $F$-continuous at a point $c \in \fhyl$ if for any $x \in (c)_F$, $f(x) \in (f(c))_F$.
\end{definition}

Just as in Luxemburg ultrapowers, this definition aims to capture in an intuitive way the idea that continuous functions are those functions $f$ for which a small change in the argument of $f$ will only entail a small change in the value of $f$. Since we want to use our notion of infinite closeness to characterize the idea of a ``small change'', we must relativize the notion of continuity on the $\fhyl$ to a viewpoint $F$. We nonetheless easily obtain the following theorem, which shows that our notion of $F$-continuity corresponds to the Weierstrass definition of standard analysis.

\begin{theorem} \label{cont}
A function $f : \RR \to \RR$ is continuous at a point $c$ if and only if $\st f$ is $F$-continuous at $c$ for any $F \in \frec$.
\end{theorem}

\begin{proof}
Assume first that $f : \RR \to \RR$ is continuous, and let $F \in \frec$, $c \in \RR$, and $x \in (c)_F$. Since $f$ is continuous at $c$, for any $n \in \omega$, there is a positive $\delta \in \RR$ such that for any $a \in \RR$, $|c - a| < \delta$ implies $ |f(c) - f(a)| < \frac{1}{n}$. Fix $n \in \omega$ and choose such a $\delta \in \RR$. Note that $F \fr |c - x| < \delta $, and thus $|| |c-x| < \delta || \in F$. By choice of $\delta$, for any $i \in || |c-x| < \delta||$, we have that $\RR \models | f(c(i)) - f(x(i)) | < \frac{1}{n}$. Thus $|| | f(c) - f(x)| < \frac{1}{n} || \in F$, and $F \fr | f(c) - f(x)| < \frac{1}{n}$. Since $n$ was chosen arbitrarily, it follows that $f(x) \in (f(c))_F$.

Conversely, suppose that there is some $F \in \frec$ such that $x \in (c)_F$ implies $f(x) \in (f(c))_F$. Note that this means that for any positive $\epsilon \in \RR$, and any positive $\delta \in (0)_F$, $F \fr \forall x (|c - x| < \delta \implies |f(c) - f(x)| < \epsilon)$. Since $(0)_F$ is non-empty, we have that $F \fr \exists \delta (\delta > 0 \me \forall x (|c-x| < \delta \implies |f(c) - f(x)| < \epsilon))$. By \cref{trans}, since this first-order sentence only has a real parameter $\epsilon$, it is also true in $\RR$. But then there must be a positive $\delta \in \RR$ such that for any $a \in \RR$, $|c - a| < \delta$ implies $|f(c) - f(a)| < \epsilon$, thus establishing that $f$ is continuous at $c$. 
\end{proof}

This establishes that one of the most appealing features of classical nonstandard analysis, an intuitive characterization of continuity that uses infinitesimals in a coherent and powerful way, is still available in our semi-constructive setting. In fact, as I will argue in \cref{canon}, there is a sense in which our non-constructive definition constitutes a better alternative to the Weierstrassian definition of continuity than the classical nonstandard one. \\

How much infinitesimal calculus could $\fhyl$ provide a foundation for? In the \textit{Epilogue} of his textbook introducing calculus via the infinitesimal method \cite{keisler2013elementary}, Keisler writes that all the results obtained in the textbook can be derived from the axioms of a complete ordered field for $\mathbb{R}$ and two axioms for the hyperreal line $\mathbb{B}^*$, the \textit{Extension axiom} and the \textit{Transfer axiom}. From the results obtained so far, it is easy to verify that the Extension axiom is valid on $\fhyl$, while the Transfer axiom is a weaker version of \cref{trans}. This means that $\fhyl$ could be used as an alternative foundation for Keisler's textbook. Because it is a possibility structure, however, some caution is in order, and some arguments need to be slightly altered. For instance, Keisler proves the following from his axiom system for $\mathbb{R}$ and $\mathbb{R}^*$:
\vspace{.5em}
\begin{description}
\item[Standard Part Principle] For any \textit{finite} hyperreal number $b$ (i.e., any $b \in \mathbb{R}^*$ such that $|b| < r$ for some $r \in \mathbb{R}$), there is exactly one real number $r$ infinitely close to $b$.
\end{description}
\vspace{.5em}
\noindent Of course, this fundamental result of nonstandard analysis is what allows for the definition of the \textit{standard part} function on the finite hyperreals, a key tool in deriving results of classical analysis. It is not true, however, that for any $F \in \frec$ and any $a \in \fhyl$ such that $F \fr |a| < \delta(r)$ for some real number $r$, we have $a \in (\delta(r))_F$ for some unique real $r$. For example, letting $a \in \RR^\omega$ be given by $a(i) = \frac{(-1)^i}{2}$, we have that $F_0 \fr -1 < a < 1$, yet clearly $a \notin (\delta(r))_{F_0}$ for any real number $r$. Moreover, there does not seem to be a principled way of determining what the standard part of $a$ should be at $F_0$, since, letting $G_1$ and $G_2$ be the filters in $\frec$ generated by adding the set of even and odd natural numbers to $F_0$, respectively, one would clearly want the $G_1$-standard part of $a$ to be $\frac{-1}{2}$ and the $G_2$-standard part of $a$ to be $\frac{1}{2}$. 

The way out of this conundrum is to ``internalize'' the Standard Part Principle, i.e., to turn it into a first-order statement in an extended language and show that it is valid on $\fhyl$. The argument, which becomes a straightforward adaptation of Keisler's, can be found in \cref{app}. The need for such a detour is essentially due to the fact that the standard part function, just like the standardness predicate, are \textit{external} objects in nonstandard analysis. Once the Standard Part Principle has been internalized, however, $\fhyl$ can play the same foundational role for Keisler's infinitesimal calculus as Luxemburg ultrapowers. Moreover, as we will see in the next section, a straightforward theory of \textit{internal} objects can also be developed for $\fhyl$.

\subsection{Internal Sets and Countable Saturation}

Internal sets play a central role in classical nonstandard analysis, as they form a collection of ``well-behaved'' subsets of a nonstandard extension of the real line. As we shall see, we can also define in a natural way the notion of an internal subset of $\fhyl$, although this must first be relativized to a point $F \in \frec$.

\begin{definition}[$F$-internal set]
Let $\{A_n\}_{n \in \omega}$ be a family of subsets of $\RR$, and let $A_F = \{a : \{i: \RR \models A_i(a(i))\} \in F\}$. An \textit{$F$-internal set} is a subset of $D$ such that $D = A_F$ for some family $\{A_n\}_{n \in \omega}$ of subsets of $\RR$.
\end{definition}

This definition mirrors the usual definition of an internal set in an ultrapower of the reals, as it can be found for example in \cite[Chap.~11]{goldblatt2012lectures}. Since $\fhyl$ is a possibility structure, however, each filter in $\frec$ gives us a different viewpoint on what an internal set determined by a countable sequence $\{A_n\}_{n \in \omega}$ looks like. To keep track of how $F$-internal sets relate to one another across different filters, it is useful to expand the language $\la$, adding a predicate symbol (which I will call an \textit{internal predicate}) $A$ for each sequence $\{A_n\}_{n \in \omega}$ of subsets of $\RR$. Let $\la^+$ be the new language. $\fhyl$ can be seen as an $\la^+$-structure by letting $a \in \itp(F,A)$ iff $||A(a)|| = \{i \in \omega : \RR \models A_i(i)\} \in F$ for any internal predicate $A$, any $a \in \fhyl$ and any $F \in \frec$. The following lemma is straightforward and ensures that this interpretation satisfies the conditions of \cref{def1}.

\begin{lemma}
Let $\{A_n\}_{n \in \omega}$ be a sequence of subsets of $\RR$ with corresponding internal predicate $A$, and $a \in \fhyl$. Then for any $F, G \in \frec:$
\begin{description}
\item[Persistence] if $G \rset F$ and $a \in \itp(F,A)$, then $a \in \itp(G,A)$;
\item[Refinement] if $a \notin \itp(F,A)$, then there is $G \rset F$ such that for all $H \rset G$, $a \notin \itp(H,A)$;
\item[Equality as congruence] if $a \in \itp(F,A)$ and $(a,b) \in \itp(F,=)$ for some $b \in \fhyl$, then $b \in \itp(F,A)$.
\end{description}

\end{lemma}

\begin{proof} Fix a countable sequence $\{A_n\}_{n \in \omega}$ of subsets of $\RR$ with associated internal predicate $A$, $a,b \in \fhyl$ and $F \in \frec$.
\begin{description}

\item[Persistence] Note that $a \in \itp(F,A)$ iff $||A(a)|| \in F \sset G$, so $a \in \itp(G,A)$.
\item[Refinement] If $a \notin \itp(F,A)$, then there is $G \in \frec$ such that $G \rset F \cup (\omega \setminus ||A(a)||)$. Clearly, for any $H \rset G$, $||A(a)|| \notin H$, from which it follows that $a \notin \itp(H,A)$.
\item[Equality as congruence] Note that $||a = b|| \cap ||A(a)|| = \{i \in \omega \RR \models a(i) = b(i) \me A_i(a(i)\} \sset ||A(b)||$. Hence $(a,b) \in \itp(F,=)$ and $a \in \itp(F,A)$ together imply that $b \in \itp(F,A)$. \qedhere
\end{description}
\end{proof}

More generally, for any $\la^+$-formula $\phi(A^1,...,A^K,\overline{x})$ with $n$ free variables where $A^1,...,A^k$ are the internal predicates appearing in $\phi$, and the corresponding sequences of subsets of $\RR$ are $\{A^1_i\sq,...,\{A^k_i\sq$, let $||\phi(A^1,...,A^k,\overline{a})|| = \{i \in \omega: \RR \models \phi(A^1_i,...,A^k_i,\overline{a(i)})\}$ for any $n$-tuple $\overline{a} \in \fhyl$. Note that this is well defined since $\phi(A^1_i,...A^k_i,\overline{x})$ is an $\la$-formula for any $i \in \omega$. By adapting the proof of \cref{los} in a straightforward way, we also obtain the following corollary.

\begin{corollary} \label{losx}
For any $\la^+$ formula $\phi(\overline{x})$ with $n$ free variables and any $n$-tuple $\overline{a}$ of elements of $\fhyl$, we have that $F \fr \phi(\overline{a})$ iff $||\phi(\overline{a})|| \in F$ for any $F \in \frec$.
\end{corollary}

This result allows us to describe the $F$-internal subsets of $\fhyl$ ``internally'', that is, using the satisfaction relation at $F$. In particular, we can define the following:

\begin{theorem}[F-Internal Set Definition Principle] \label{indef}
Let $\phi(\overline{A},\overline{x},y)$ be an $\la^+$ formula with $n+1$ variables and where $\overline{A}$ is a tuple of internal predicates. Then for any $n$-tuple $\overline{a} \in \fhyl$ and any $F \in \frec$, $\{b \in \fhyl : F \fr \phi(\overline{A},\overline{x},b)\}$ is an $F$-internal subset of $\fhyl$.  
\end{theorem}

\begin{proof}
By induction on the complexity of $\phi$. Note that this is trivial if $\phi(x) = A(x)$ for some internal predicate $A$.
\begin{itemize}
    \item If $R$ is an $n+1$-ary $\la$ relation symbol, then let $\{R_i\sq$ be such that $R_i = \{b \in \RR : R(\overline{a}(i),b)\}$, and let $\st R^{\overline{a}}$ be the corresponding internal predicate in $\la^+$. Then clearly for any $F \in \frec$ $F \fr \st R(\overline{a},b)$ iff $||R(\overline{a},b)|| \in F$ iff $||\st R^{\overline{a}}(b)|| \in F$ iff $b \in \st R^{\overline{a}}_F$. 
    \item If $\phi := \neg \psi (\overline{A},\overline{x},y)$, then by induction hypothesis we have a countable sequence $\{S^\psi_i\sq$ such that $\{b \in \fhyl : F \fr \psi(\overline{A},\overline{a},b)\} = S^\psi_F$ for any $F \in \frec$. Define $\{S^\phi_i\sq$ by letting $S^\phi_i = \RR \setminus S^\psi_i$ for any $i \in \omega$. Then, for any $b \in \fhyl$, we have that $||S^\phi(b)|| = \omega \setminus ||S^\psi(b)||$, from which it follows that $\{b \in \fhyl : F \fr \phi(\overline{A},\overline{a},b)\} = S^\phi_F$.
    \item If $\phi := \psi \me \chi(\overline{A},\overline{x},y)$, then by induction hypothesis we have countable sequences $\{S^\psi_i\sq$ and $\{S^\chi_i\sq$ such that $\{b \in \fhyl : F \fr \psi(\overline{A},\overline{a},b)\} = S^\psi_F$ and $\{b \in \fhyl : F \fr \chi(\overline{A},\overline{a},b)\} = S^\chi_F$ for any $F \in \frec$. Define $\{S^\phi_i\sq$ by letting $S^\phi_i = S^\psi_i \cap S^\chi_i$ for any $i \in \omega$, and note that $||S^\phi(b)|| = ||S^\psi(b)|| \cap ||S^\chi(b)||$ for any $b\in \fhyl$, thus showing that $\{b \in \fhyl : F \fr \phi(\overline{A},\overline{a},b)\} = S^\phi_F$.
    \item If $\phi := \exists z \psi(\overline{A},\overline{x},y,z)$, then by induction hypothesis we have countable sequences $\{S^{\psi(c)}_i\sq$ such that $\{b \in \fhyl : F \fr \psi(\overline{A},\overline{a},b,c)\} = S^{\psi(c)}_F$ for any $c \in \fhyl$. For any $i \in \omega$, let $S^\phi_i = \bigcup_{c \in \fhyl} S^{\psi(c)}_i$. I claim that $|| \exists z \psi(\overline{A},\overline{a},b,z)|| = ||S^\phi(b)||$ for any $b \in \fhyl$, which is enough to show that $\{b \in \fhyl : F \fr \phi(b)\} = S^\phi_F$. For the proof of the claim, note first that, using the Axiom of Countable Choices, we can define $c \in R^\omega$ such that $||\exists z \psi(\overline{A},\overline{a},b,z|| \sset ||\psi(\overline{A},\overline{a},b,c|| = ||S^{\psi(c)}(b)|| \sset ||S^\phi(b)||$. Finally, the converse direction is straightforward: if $i \in ||S^\phi(b)||$, then $\RR \models \psi(\overline{A_i},\overline{a(i)},b(i),c(i))$ for some $c \in \fhyl$, which implies that $\RR \models \exists z \phi(\overline{A_i},\overline{a(i)},b(i),z)$, and hence $i \in ||\exists x \phi(\overline{A},\overline{a},b,c)||$. \qedhere
\end{itemize}
\end{proof}

As an immediate corollary of \cref{indef}, we obtain the following result on the structure of $F$-internal sets.

\begin{corollary}
Let $F \in \frec$. Every nonstandard extension of a subset of the reals is an $F$-internal set. Moreover, the $F$-internal sets form a Boolean algebra.
\end{corollary}

\begin{proof}
Recall that if $S \sset \mathbb{R}$, then $\st S_F = \{b \in \fhyl : F \fr S(b)\}$. By \cref{indef}, it follows at once that $\st S_F$ is $F$-internal for any $F \in \frec$. Moreover, for any internal predicates $A, B$, let $\neg_F A_F = \{b \in \fhyl : F \fr \neg A(b)\}$, $A_F \me_F B_F = \{ b \in \fhyl : F \fr A(b) \me B(b)\}$, and $A_F\jo_F B_F = \{b \in \fhyl : F \fr A(b) \jo B(b)\}$. By \cref{indef}, the operations $\neg_F$, $\me_F$ and $\jo_F$ map $F$-internal sets to $F$-internal sets, and, using \cref{losx}, it is straightforward to verify that $\mathfrak{I}_F = (\{A_F\}_{A \in \la^+ \setminus \la}, \me_F, \jo_F, \neg_F, \fhyl, \emptyset)$ is a Boolean algebra.
\end{proof}

The Definition Principle above therefore gives us a powerful way of defining $F$-internal sets and ensures that many subsets of $\fhyl$ are $F$-internal for any $F$. Note however that, because we are working in a possibility structure, the extension of an internal subset of $\fhyl$ must always be relativized to a filter $F$. In other words, the Definition Principle ensures that \textit{the same} internal subsets, understood \textit{intensionally}, exist at any point $F \in \frec$. But whether an element in $\fhyl$ belongs to a given internal set $S$ will vary with the filter $F$, i.e., the \textit{extension} of internal sets is relative to the viewpoints in $\frec$. To illustrate this point, consider the internal predicate $\st N$, corresponding to the sequence $\{N_i\sq$ where $N_i$ is the set of all natural numbers for any $i \in \omega$. In other words, $\st N$ represents the nonstandard extension of the natural numbers, i.e., the internal set of hypernatural numbers. For any $F \in \frec$, the hypernatural-numbers-at-$F$ are an $F$-internal set, namely $\st N_F$. But this does not mean that the set of hypernatural numbers is stable across points in $\frec$. As a matter of fact, if we consider the sequence $a \in \fhyl$ defined by $a(2k) = k$ and $a(2k+1) = \pi$ for all $k \in \omega$, then $a \in \st N_F$ precisely when $\{2k : k \in \omega\} \in F$. From now on, I will adopt the ``intensional'' perspective on internal sets whenever possible, meaning that I will use the phrase ``internal set'' to designate the varying extension of an internal predicate $S$.\\

In Luxemburg ultrapowers, a crucial feature of internal sets is their countable saturation property: any countable decreasing sequence of non-empty internal sets has a non-empty intersection. As we shall see, the same property, once relativized in a natural way to viewpoints in $\frec$, can also be proved for the internal sets on $\fhyl$.

\begin{notation}
Let $A$ and $B$ be internal predicates corresponding to countable sequences $\{A_i\sq$ and $\{B_i\sq$ respectively. Then $||A \sset B||$ is defined as the set $\{i \in \omega : A_i \sset B_i\}$, and $||A \neq \emptyset||$ as the set $\{i \in \omega : A_i \neq \emptyset\}$.
\end{notation}

By \cref{losx}, it is straightforward to verify that for any $F \in \frec$, and any internal predicates $A$ and $B$, $||A \sset B|| \in F$ iff $F \fr \forall x (A(x) \to B(x))$, and $||A \neq \emptyset || \in F$ iff $F \fr \exists x \in A$. We are now in a position to prove the second important feature of $\fhyl$:

\begin{theorem}[Countable F-saturation Principle] \label{sat}
Let $F \in \frec$ and let $\{X^i\sq$ be a family of $F$-internal sets such that for any $k \in \omega$, $F \Vdash \exists x X^k(x)$ and $F \Vdash \forall x(X^{k+1}(x) \to X^k(x))$. Then there is $a \in \fhyl$ such that $F \Vdash  X^k(a)$ for all $k \in \omega$.
\end{theorem}

\begin{proof} The proof mirrors the usual proof for Luxemburg ultrapowers as it can be found for example in \cite[Theorem 11.10.1]{goldblatt2012lectures}. For any $k \in \omega$, let $\{A^k_n\}_{n \in \omega}$ be the countable sequence of subsets of $\mathbb{R}$ such that $X^k = A^k_F$. Since $F \Vdash \exists x X^k(x)$, this means that $||A^k \neq \emptyset|| \in F$ for each $k \in \omega$. Similarly, since $F \Vdash \forall x (X^{k+1}(x) \to X^k(x))$ for any $k \in \omega$, this means that $||A^{k+1} \sset A^k|| \in F$ for all $k$.
Hence for any $k \in \omega$, letting $J^k = \cap_{i < k} ||A^{i+1} \sset A^i|| \cap ||A^k \neq \emptyset||$, we have that $J^k \in F$ and $J^i \sset J^k$ for any $i \geq k$. Now if we construct $a \in \fhyl$ such that for any $k \in \omega$ and any $n \geq k$, if $n \in J^k$, then $a(n) \in A^k_n$, then as $\{n \in \omega : k \leq n\} \cap J^k \in F$ since $F$ extends the Fréchet filter on $\omega$, we will have that $a \in A^k_F$ for all $a \in F$. We define $a$ in the standard way: for any $n \in J
^1$, let $k_n = max\{k : k \leq n \text{ and } n \in J^k\}$. Clearly, $n \in J^{k_n}$, so by definition $n \in ||A^{k_n} \neq \emptyset||$, i.e., $\mathbb{R} \models \exists x A_n^{k_n}(x)$ and moreover, $A^{j+1}_n \sset A^j_n$ for any $j < k$. So let $a(n) \in A^{k_n}$ if $n \in J^1$, and let $a(n)$ be arbitrary otherwise. Finally, fix $k \in \omega$, and note that, if $n \geq k$ and $n \in J^k$, then $k \leq k_n$, hence $a(n) \in A_n
^{k_n} \sset A_n^k$. Thus $a \in \bigcap_{k \in \omega} (X^k)_F$.
\end{proof}

Let us conclude this section with some results on the cardinality of $F$-internal sets. Of course, since our internal notion of equality is an equivalence relation relative to a point $F \in \frec$, we must be careful in defining internally what finiteness, countability and uncountability mean for internal sets. For example, we would clearly want the set of natural numbers less than some positive natural number $n$ to be finite. But it is easy to see that for any $F \in \frec$, there are at least countably many $a \in \fhyl$ for which $F \fr a < n$, since for any $k \in \omega$, the function $a_k : \omega \to \RR$ given by $a_k(i) = 0$ if $i < k$ and $a_k(i) = n-1$ otherwise will be one such element of $\fhyl$. This motivates the following definitions:

\begin{definition}
Let $F \in \frec$. A subset $A$ of $\fhyl$ is \textit{$F$-finite} if $A_F = \bigcup_{i \in I} a^i_F$ for some finite set $I$ with $a^i \in \fhyl$ for all $i \in I$, \textit{$F$-countable} if there is a countable sequence $\{a^i\sq$ of elements of $\fhyl$ such that $A = \bigcup_{i \in \omega} a^i_F$, and \textit{$F$-uncountable} otherwise.
\end{definition}

It is clear from the previous definition that sets such as the set of all natural numbers below some positive number $n$ are $F$-finite for any $F \in \frec$, just as the set of standard natural numbers is $F$-countable for any $F$. However, the following lemma also shows that our relativized notions of finiteness and countability interact with the notion of an internal set in a meaningful way.
\begin{lemma}
\begin{enumerate}
\item[]
    \item Let $A$ be a subset of $\fhyl$ and $F \in \frec$ be such that for all $a \in \fhyl$, $a \in A$ iff there is $r \in \RR$ such that $a \in \delta(r)_F$. Then $A$ is $F$-internal iff it is $F$-finite.
    \item Any internal subset of $\fhyl$ is either finite or uncountable.
\end{enumerate}

\end{lemma}

\begin{proof}
\begin{enumerate} 
    \item[]
    \item Let $B \sset \RR$ be the set $\{r \in \RR : \delta(r) \in A\}$. By assumption, we have that $A = \bigcup_{r \in B} \delta(r)_F$. Note that $A$ is $F$-finite iff $B$ is finite. Clearly, if $B = \{r_1,...,r_n\}$ is finite, then for all $a \in \fhyl$, $a \in A$ iff $F \fr a = \delta(r_1)  \jo \cdots \jo a = \delta(r_n)$, hence $A$ is $F$-internal by \cref{indef}. Conversely, suppose that $A$ is $F$-internal, i.e., that there is a sequence $\{B_i\sq$ such that $B_F = A$. Assume $B$ is infinite, and let $\{r^i\}_{i \in \omega}$ be a countable sequence of elements of $B$.\footnote{Note that we are using the Axiom of Countable Choices here to ensure that the infinite set $B$ is also Dedekind-infinite.} Let $J$ be the set of all $i \in \omega$ such that $B_i$ is infinite, and relabel $J$ and $\omega \setminus J$ as $\{n_i\sq$ and $\{m_i\sq$ respectively. Define $a : \omega \to \RR$ as follows. For any $i \in \omega$, $a(m_i) = r^m$ for $m$ largest such that $r_m \in B_{m_i}$ (or some fixed $r \notin B$ if $B_{m_i}$ is empty), and $a(n_i) = r^n$ for $n$ smallest such that $r_n \in B_{n_i} \setminus \{a(n_j) : j < i\}$. By construction, $i \in ||B_i(a)||$ whenever $B_i \neq \emptyset$, hence $a \in B_F$. Moreover, for any $i \in \omega$, $|||a = r^i|| \cap J| \leq 1$, and therefore $\omega \setminus (||a = r^i|| \cap J) \in F$. Note also that $j \in ||a = r^i|| \setminus J$ implies that $|B_j| \leq i$. But since $F \fr \exists x_1,...,x_{i+1} \bigme_{k\neq l\leq i+1}x_k \neq x_l \me \bigme_{k \leq i+1} B(x_k)$, it follows that $||a = r^i|| \setminus J \notin F$. But this implies that $||a = r^i|| \notin F$ for any $i \in \omega$, and hence that $a \in B_F \setminus A$, contradicting our hypothesis. Hence if $A$ is $F$-internal, it must also be finite. 
    \item Let $A$ be an internal subset of $\fhyl$ with $\{A_i\sq$ its associated countable sequence of subsets of $\RR$, and suppose that $A$ is $F$-countable. This means that we have a sequence $\{a^i\sq$ of elements of $\fhyl$ such that $A_F = \bigcup_{i \in \omega} a^i_F$. By the Definition Principle, for any $n \in \omega$, the set $A^i = A \setminus \{a^1,\dotsc.a^i\}$ is $F$-internal. But if $F \fr \exists x A^i(x)$ for all $i \in \omega$, then by \cref{sat} there is $a \in \fhyl$ such that $F \fr A^i(a)$ for all $i \in \omega$, a contradiction. Therefore $A^k_F$ must be empty for some $k \in \omega$, which implies that $A_F = \bigcup_{i \leq k} a^k_F$, and thus that $A$ is $F$-finite. Hence any $F$-internal set is either $F$-finite or uncountable. \qedhere
\end{enumerate}
\end{proof}

The previous lemma echoes a well-known result for Luxemburg ultrapowers, namely that the only internal subsets of $\RR$ on the hyperreal line are the finite ones and that any internal subset of the hyperreal line is either finite or uncountable \cite[Sections 11.7 and 11.12]{goldblatt2012lectures}. As an immediate consequence, we have that the set of all standard natural numbers $\mathcal{N}$ and the set of all standard real numbers $\RR$ are always $F$-external.\\

Let us sum up our progress so far. As we have seen in this section, the $F$-hyperreal line $\fhyl$ is a structure that shares many important features of the classical hyperreal line, including versions of the Transfer and Saturation Principles, infinitesimals, and a robust notion of internal set. The crucial difference though is that, while the classical hyperreal is a Tarskian structure whose existence requires strong non-constructive hypotheses like the Ultrafilter Lemma, $\fhyl$ is a possibility structure which can be defined in a semi-constructive setting, assuming only $ZF+DC$. As mentioned in \cref{intro}, the idea of proposing more constructive frameworks for nonstandard analysis is far from new and is in fact arguably as old as Robinsonian nonstandard analysis itself \cite{schmieden1958erweiterung}. As we will see in the next three sections, $\fhyl$ can be thought of as an ``intersection point'' of three such research programs: the use of the Fr\'echet filter in place of a non-principal ultrafilter on $\omega$, the use of sheaves instead of sets as the carrier of a model, and the use of Boolean-valued instead of classical, two-valued, Tarskian models. As we will see, the comparison of $\fhyl$ with each such proposal in the literature also allows for different perspectives on $\fhyl$ and on its relationship with the classical hyperreal line.

\section{The Asymptotic Approach} \label{asym}
The first alternative approach to the hyperreal line that I will discuss aims to do away with the non-constructive aspect of Robinsonian nonstandard analysis by substituting the Fr\'echet filter for a non-principal ultrafilter on $\omega$. Although constructive, the structure thus obtained, known as a reduced power, has some significant drawbacks, most importantly the lack of a Transfer Principle. On the other hand, according to its proponents, this idea has the advantage of not introducing in analysis an ``exotic'' object like a non-principal ultrafilter---an \textit{intangible}, to use Schechter's terminology in \cite{schechter1996handbook}---and appears to be closer both to ordinary mathematical practice and to historical work that involves reasoning about mathematical objects ``in the limit''. I will first present Laugwitz's work in \cite{laugwitz}, before discussing a recent proposal by Tao in \cite{taocheapnsa}, which can be seen as a refinement of Laugwitz's idea, and comparing each proposal to $\fhyl$.

\subsection{Laugwitz and Reduced Powers} \label{laug}

Laugwitz's starting point is slightly different from what we have discussed so far. Indeed, while we have been taking the Cantor-Dedekind real line for granted and exploring how to expand it to a hyperreal line, Laugwitz is first interested in exploring the kind of structures that can be obtained from sequences of rationals. Drawing a contrast with Cantor's construction of the reals as equivalence classes of Cauchy-converging sequences of rationals, he writes:

\begin{quote}
    The two steps of the sequential approach to the real numbers are (i) to restrict the set of admissible rational sequences to ``fundamental'' sequences, and (ii) to furnish this set of fundamental sequences with an equivalence relation. Since our extended number system is expected to contain ``more'' numbers than $\mathbf{R}$ itself, we have to change these assumptions, and there are two possibilities to achieve this aim: (i) to admit more sequences, (ii) to relax the equivalence relation. In a radical way, we shall (i) admit all sequences, and (ii) we shall identify only those sequences which are equal for almost all numbers $n$, $a \equiv b$ iff the complement of $\{n \in \mathbf{N} | a_n = b_n\}$ is finite. \cite[p.~11]{laugwitz}
\end{quote}

Formally, Laugwitz works with a reduced power of a field $(K,+,\cdot,0,1)$ modulo the Fr\'echet filter $F_0$. The domain of such a structure, noted $^\Omega K$ by Laugwitz, is given by equivalence classes of sequences $a : \omega \to K$, where two such sequences $a$ and $b$ are equivalent if and only if they agree on cofinitely many values of $n$. Operations on $K$ can be extended to $^\Omega K$ pointwise: for example, given two sequences $a$, $b$ with equivalence classes $a_{F_0}$ and $b_{F_0}$, $a_{F_0} + b_{F_0}$ is the equivalence class of the function $c : \omega \to K$ defined by $c(n) = a(n)+b(n)$ for every $n \in \omega$. The properties of a filter are enough to guarantee that $(^\Omega K, +, \cdot)$ is a well-defined first-order structure. But, as Laugwitz notes, working with a reduced power has some important drawbacks:

\begin{quote}
    It is an easy task to show that $(^\Omega K,+,\cdot)$ is a commutative ring with unit element $[1]$ [$\dotsc$]. Unfortunately $^\Omega K$ fails to be a field. Consider the sequences 
    \begin{align}
        a_n = 1 + (-1)^n, \hspace{1em} b_n = 1 - (-1)^n \label{3}
    \end{align}
    then $\alpha \neq 0$ and $\beta \neq 0$ since neither the set of odd numbers nor the set of even numbers is cofinite. But, obviously, $\alpha \cdot \beta =0$. So $(^\Omega K, +, \cdot)$ is a ring with divisors of zero and certainly not a field. Moreover, the generalized numbers $\alpha$, $\beta$ defined by (\ref{3}) are neither positive nor negative, which means that our ring is only partially ordered by $<$. [p.~13]
\end{quote}

Of course, this is an immediate consequence of the use of the Fr\'echet filter rather than a non-principal ultrafilter: while a weaker version of \L\'os's Theorem, the preservation in the reduced power of sentences true in the base structure, holds for a fairly large class of first-order sentences (namely Horn sentences, see \cite{chang1990model}) that includes equations, it does not hold for all first-order sentences, particularly sentences involving disjunctions and negations. This means that the ``expanded real line'' that Laugwitz is considering does not have the structure of a field and is not even a linear order. Interestingly, Laugwitz argues that this feature is not necessarily an issue, by drawing a parallel to the fact that the complex numbers are a domain extension of the ordered field of reals that is not itself an ordered field:

\begin{quote}
    Now this situation is not surprising. We are used to the fact that not all good properties of a number system extend automatically to a larger system. For instance, the order structure of $\mathbf{R}$ cannot be extended to the complex numbers. But nevertheless we reach the essential goal of the complex numbers, to solve algebraic equations. Since we aimed at introducing infinitely large numbers as well as infinitesimals we shall look for elements of $^\Omega K$ enjoying these properties. [p.~13]
\end{quote}

Laugwitz's argument seems to be that working in a structure that is neither a field nor a linear order is an acceptable price to pay in order to have infinitesimal and infinitely large numbers, just like working in a non linearly-ordered field is an acceptable price to pay to obtain a field that is algebraically closed. Of course, much could be said about this analogy. Although no algebraically closed field can be linearly ordered, since it follows from the ordered field axioms that $x^2 > 0$ whenever $x < 0$ or $x > 0$, it is not true that a linearly ordered field must always be Archimedean: the hyperreal line is a counterexample, but so are Conway's surreal numbers, or the Levi-Civita field \cite{ehrlich2013real}. As we have seen in the previous section, $\fhyl$ is also a possibility structure on which all the axioms of a field are valid.

Nevertheless, Laugwitz goes on to show how his framework allows him to rigorously reproduce several historical arguments from Leibniz, Euler, and Cauchy, among others. His goal is to show, quite convincingly, that the reduced power construction captures a powerful and historically significant intuition that the behavior of sequences ``in the limit'' is determined by their behavior ``almost everywhere'', i.e., on cofinitely many of their values:

\begin{quote}
    I should like to emphasize that our rather simple definitions already permit correct interpretations of large parts of the history of infinitesimal mathematics as well as of the use of divergent expressions. [...] The proofs invariably run as follows: Write $n$ instead of $^\Omega$ and show that everything holds for almost all $n \in \mathbf{N}$! [p.~14]
\end{quote}

Finally, Laugwitz shows how a system isomorphic to the Cantor-Dedekind real numbers can be obtained as a quotient of the structure $^\Omega \mathbb{Q}$. His notation is at times a bit obscure, but the main idea is to define a \textit{cut number} in $^\Omega \mathbb{Q}$ as a sequence $a: \omega \to \mathbb{Q}$ such that the absolute value of $a(n)$ is bounded by some rational $q$ cofinitely often, and for any rational number $r$, $|a(n) - r|$ is smaller than any fraction $\frac{1}{m}$ for any ${m \in \omega}$ for some large enough $n$ (in which case $a$ is infinitesimally close to $r$), or $a(n) < r$ cofinitely often, or $a(n) > r$ cofinitely often. In other words, a cut number is a bounded element $a_{F_0}$ of the reduced power of $\mathbb{Q}$ modulo the Fr\'echet filter for which a form of the trichotomy law holds with respect to the rationals: for all $q \in \mathbb{Q}$, either $^\Omega \mathbb{Q} \models a_{F_0} < q \jo a_{F_0} > q$, or $a_{F_0}$ is in the \textit{monad} of $q$ (denoted $\text{mon }  q$) i.e., $^\Omega \mathbb{Q} \models |a_{F_0} - q| < \frac{1}{m}$ for every natural number $m$.\footnote{Laugwitz's \textit{monads} correspond to what we called \textit{halos} in \cref{halo}, following \cite{goldblatt2012lectures}.} Laugwitz goes on to prove that the cut numbers form a subring $C$ of the bounded numbers of $^\Omega \mathbb{Q}$, and that mon $0$, i.e., the set of numbers infinitesimally close to $0$, is a maximal ideal $C$. Taking the quotient $C/\text{mon } 0$ thus yields a field, which can then be shown to be isomorphic to the Cantor-Dedekind reals. This construction of the reals from the hyperrational numbers can be carried out in a similar fashion in classical nonstandard analysis, by just quotienting the ring of bounded hyperrationals by the relation of infinitesimal closeness. Laugwitz's setting however requires slightly more work, since the ring $^\Omega \mathbb{Q}$ is not totally ordered. Laugwitz uses this result to give an intuitive interpretation of the cut numbers, comparing once more the role they can play in calculus to the role that the complex numbers play in algebra:

\begin{quote}
    We are now in a position to give some intuitive interpretations of our concepts. Though the cut numbers themselves are not a totally ordered set their equivalence classes or monads enjoy this property. We can imagine them as clusters like pearls on a string.
    Of course between each two of these clusters there are others.[...] The elements of these clusters can be regarded as representatives of the fine structure of real numbers. [...] Though these numbers lack a representation in the geometrical continuum we use them as calculation devices just as everybody uses complex numbers to get real results more easily. [p.~18]
\end{quote}

To sum up, Laugwitz works with a reduced power of a field $K$ modulo the Fr\'echet filter in order to introduce new infinite and infinitesimal numbers in $K$. He shows how many classical results about infinite series and calculus can be derived in an elegant fashion in this framework and how the reduced power of $\mathbb{Q}$ allows for an alternative construction of the Cantor-Dedekind reals. The major drawback of this approach, of course, is the lack of a Transfer Principle between the field $K$ and the reduced power $^\Omega K$, which fails to be a field and is not totally ordered. It is worth mentioning that, in some precise sense, the $F$-hyperreal line $\fhyl$ contains a copy of the reduced power of $\RR$ modulo the Fr\'echet filter $F_0$. Indeed, taking the quotient of $\RR^\omega$ modulo the relation $\sim_{F_0} = \{(a,b) \in \RR^\omega : F_0 \fr a = b\}$, and interpreting any $\la$-relation symbol $R(\overline{x})$ by $\{\overline{a_{F_0}} \in (\RR^\omega)_{/\sim_{F_0}} : F_0 \fr R(\overline{a})\}$, one obtains exactly Laugwitz's $^\Omega\RR$. Note however that this does not imply in any way that the forcing relation at $F_0$ coincides with satisfaction in $^\Omega\RR$, since the Transfer Principle guarantees that the axioms of an ordered field are forced at any $F \in \frec$. The difference comes down to the recursive clauses of the forcing conditions in possibility semantics, which, unlike the recursive clauses in Tarskian semantics, allow us to consider not only satisfaction of atomic formulas at a reduced power of $\RR$ modulo the Fr\'echet filter but also at any other reduced power of $\RR$ modulo a non-principal ultrafilter on $\omega$.

Since validity on $\fhyl$ coincides with the forcing relation at $F_0$, we may therefore argue that $\fhyl$, just like Laugwitz's framework, formalizes the idea that many classical arguments from calculus rely on the intuition that the behavior of mathematical objects ``in the limit'' is determined by how often they exhibit such a behavior in the finite and that this intuition can be given a rigorous and elegant presentation in a structure that contains infinitely large and infinitely small elements. But where Laugwitz weakens the requirements on the axioms that such an enlarged structure should satisfy (requiring it to be merely a ring rather than an ordered field), we weaken the requirement that this enlarged structure be a classical, Tarskian structure and construct it instead as a possibility structure that satisfies all the axioms of ordered fields. This trade-off between standard semantics and powerful axiomatics appears also in Tao's proposal of a ``cheap'' nonstandard analysis, which shares many features with Laugwitz's approach, but can also be thought of as an informal version of $\fhyl$. 

\subsection{Tao and ``cheap'' Nonstandard Analysis} \label{tao}

In \cite{taocheapnsa}, Tao introduces what he calls ``cheap'' nonstandard analysis as an informal framework that allows one to apply some of the elegant reasoning of nonstandard analysis, without relying on the existence of non-principal ultrafilters and the construction of ultrapowers. Tao's motivation seems both foundational, as the use of infinitary methods to obtain ``finitary'' results seems to violate some purity concern, as well as purely technical:

\begin{quote}
    On the other hand, nonprincipal ultrafilters do have some unappealing features. [...O]ne cannot actually write down any explicit example of a nonprincipal ultrafilter, but must instead rely on nonconstructive tools such as Zorn’s lemma, [...] or the [Ultrafilter Lemma] to locate one. As such, ultrafilters definitely belong to the “infinitary” side of mathematics, and one may feel that it is inappropriate to use such tools for “finitary” mathematical applications, such as those which arise in hard analysis. From a more practical viewpoint, because of the presence of the infinitary ultrafilter, it can be quite difficult [...] to take a finitary result proven via nonstandard analysis and coax an effective quantitative bound from it.
\end{quote}

Tao's suggestion is very close to Laugwitz's idea, as it involves working with the Fr\'echet filter rather than with a non-principal ultrafilter on $\omega$. Tao also seems to think that, unlike non-principal ultrafilters, the Fr\'echet filter is not an ``exotic'' mathematical object but rather already plays a key role in any kind of ordinary mathematical reasoning that involves passing to a limit:

\begin{quote}
    There is however a “cheap” version of nonstandard analysis which is less powerful than the full version, but is not as infinitary in that it is constructive [...]. It is obtained by replacing the nonprincipal ultrafilter in fully nonstandard analysis with the more classical Fréchet filter of cofinite subsets of the natural numbers, which is the filter that implicitly underlies the concept of the classical limit ${\lim_{{\bf n} \rightarrow \infty} a_{\bf n}}$ of a sequence when the underlying asymptotic parameter ${{\bf n}}$ goes off to infinity. As such, “cheap nonstandard analysis” aligns very well with traditional mathematics [...]. 
\end{quote}

More precisely, the starting point of Tao's proposal is to consider mathematical objects that may vary along some parameter $\bf{n}$, which ranges over $\omega$. Such a mathematical object $x_{\bf{n}}$ can therefore be thought of as a countable sequence of classical objects. The mathematical universe of Tao's cheap nonstandard analysis is therefore two-sorted: there are \textit{standard} or classical mathematical objects, which do not depend on the parameter $\bf{n}$, and \textit{nonstandard} ones, which may vary with $\bf{n}$. Of course, we may always think of standard objects as nonstandard ones that happen to not vary  with $\bf{n}$. Two nonstandard objects $x_{\bf{n}}$ and $y_{\bf{n}}$ are considered equal if they differ on only finitely many values of the parameter $\bf{n}$. More generally, a tuple of nonstandard objects $\overline{x_{\bf{n}}}$ stand in a relation $R$ if and only if they stand in relation $R$ for cofinitely many values of the parameter $\bf{n}$. From this, it follows pretty clearly that Tao's approach amounts to working with equivalence classes of the countable sequences of objects modulo the Fr\'echet filter and is thus very similar to Laugwitz's proposal. As Tao himself remarks, it has similar shortcomings:

\begin{quote}
    The catch is that the Fréchet filter is merely a filter and not an ultrafilter, and as such some of the key features of fully nonstandard analysis are lost. Most notably, the law of the excluded middle does not transfer over perfectly from standard analysis to cheap nonstandard analysis [...]. The loss of such a fundamental law of mathematical reasoning may seem like a major disadvantage for cheap nonstandard analysis, and it does indeed make cheap nonstandard analysis somewhat weaker than fully nonstandard analysis.
\end{quote}

Tao's framework therefore not only does not satisfy the Transfer Principle but does not even allow for classical reasoning. Tao argues that the problem can often be bypassed by reasoning intuitionistically and that a large class of formulas transfer from the standard universe to the nonstandard universe, although he does not precisely isolate a fragment of the first-order language for which the Transfer Principle would remain valid. Moreover, he also discusses a way of retrieving some form of the law of excluded middle, by using a technique that he calls ``passing to a subsequence'':

\begin{quote} Furthermore, the law of the excluded middle can be recovered by adopting the freedom to pass to subsequences with regards to the asymptotic parameter ${{\bf n}}$; this technique is already in widespread use in the analysis of partial differential equations, although it is generally referred to by names such as “the compactness method” rather than as “cheap nonstandard analysis”.
\end{quote}

What does Tao mean by ``passing to a subsequence'' here? His treatment of the notion remains informal, but the idea is that a given mathematical object $x_{\mathbf{n}}$, indexed by some parameter $\bf{n}$ that ranges over $\omega$, may be identified with another mathematical object $x'_{\mathbf{n}}$, in which the parameter $\bf{n}$ now only ranges over some proper infinite subset $\Sigma \sset \omega$, and $x'(n) \neq x(n)$ for only finitely many $n \in \Sigma$. The example he gives is that of two parametrized reals $x_{\mathbf{n}}$ and $y_{\mathbf{n}}$ such that $x_{\mathbf{n}}y_{\mathbf{n}} = 0$. Clearly, this does not imply that $x_{\mathbf{n}} = 0$ or $y_{\mathbf{n}} = 0$, but it implies that there must be some infinite $\Sigma \sset \omega$ such that $x(n) = 0$ for all $n \in \Sigma$ or $y(n) = 0$ for all $n \in \Sigma$, since $\{n \in \omega :x(n)y(n) = 0\} = \{n \in \omega : x(n) = 0\} \cup \{n \in \omega : y(n)=0\} $ is infinite. Therefore, by ``passing to the subsequence'' parametrized by $n$ ranging over $\Sigma$, it follows that $x_{\mathbf{n}} = 0$ or $y_{\mathbf{n}} = 0$. Now let $F_\Sigma$ be the filter on $\omega$ extending the Fr\'echet filter and generated by $\Sigma$. Clearly, ``passing to the subsequence'' determined by $\Sigma$ amounts to declaring that $x_{\mathbf{n}}$ or $y_{\mathbf{n}}$ is equivalent to $0$, provided we weaken the notion of equivalence between two parametrized objects from being equal on a cofinite set to being equivalent on a set in $F_\Sigma$. Thus Tao's technique of passing to a subsequence in cheap nonstandard analysis can be identified with our use in $\fhyl$ of the entire poset $\frec$ of non-principal filters on $\omega$, since reducing the range of the parameter $\bf{n}$ amounts to identifying together more parametrized objects, just like more elements of $\fhyl$ are forced to be equal at a filter $G$ than at a filter $F$ when $G \rset F$. More evidence that Tao's cheap nonstandard analysis can be meaningfully interpreted as an informal version of $\fhyl$ comes from his remark that ``in cheap nonstandard analysis one only works with statements which remain valid under the operation of restricting the underlying domain of the asymptotic parameter'', which is analogous to the \textit{persistence} condition imposed on formulas in possibility semantics, and from what Tao dubs a version of the Saturation Principle:

\begin{quote}Let us call a nonstandard property $P(x)$ pertaining to a nonstandard object $x$ satisfiable if, no matter how one restricts the domain $\Sigma$ from its current value, one can find an $x$ for which $P(x)$ becomes true, possibly after a further restriction of the domain $\Sigma$. The countable saturation property is then the assertion that if one has a countable sequence $P_1, P_2, \ldots$ of properties, such that $P_1(x),\ldots,P_n(x)$ is jointly satisfiable for any finite $n$, then the entire sequence $P_1(x), P_2(x), \ldots$ is simultaneously satisfiable.
\end{quote}

The analogy between a parameter space $\Sigma$ and the non-principal filter $F_\Sigma$ on $\omega$ generated by the finite set $\Sigma$ on one hand, and between passing to a subsequence $\Sigma'$ and extending the filter to one containing $\Sigma'$, allows us to see the similarity between Tao's principle and \cref{sat}, as Tao's notion of satisfiability of a property $P(x)$ on a parameter space $\Sigma$ becomes quite transparently analogous to the formula $\exists x P(x)$ being forced at $F_\Sigma$.\\

To sum up, Tao's cheap nonstandard analysis has the same starting point as Laugwitz's proposal, namely considering objects that vary along a parameter $\bf{n}$ and whose properties are exactly those properties that these varying objects have on cofinitely many of their values. The main drawback of such a framework is that it does not satisfy the law of excluded middle and therefore lacks the powerful Transfer Principle of classical nonstandard analysis. As we have argued, $\fhyl$ stems from a similar idea but solves the Transfer Principle problem thanks to the machinery of possibility semantics, which also allows one to reason classically within the structure. Although it remains informal, Tao's proposal of having the freedom to pass to a subsequence by varying the range of the parameter $\bf{n}$ can conceivably be seen as analogous to the freedom to vary the non-principal filter on $\omega$ that is afforded by possibility semantics. In this sense, $\fhyl$ can a seen as formal way to capture Tao's proposal, with the added benefit of having a powerful version of the Transfer Principle. Interestingly, Tao himself argues that ``[t]he dynamic nature of the parameter space $\Sigma$ makes it a little tricky to properly formalise cheap nonstandard analysis in the usual static framework of mathematical logic''. This seems to suggest that the central feature of $\fhyl$ that allows it to formalize Tao's idea lies in its more ``dynamic'' nature, i.e., in the fact that it is able to reason across various classical models rather than just within one. As we shall see in the next section, this ``dynamic'' character of $\fhyl$ can be captured quite precisely, by comparing it to another constructive version of nonstandard analysis that uses sheaf semantics and topos theory.

\section{The Dynamic Approach} \label{dyn}

The second kind of work I will be considering approaches nonstandard analysis from a constructive and intuitionistic perspective. There are, of course, many different constructive or quasi-constructive approaches to analysis, which differ more or less dramatically from classical analysis, such as Bishop's constructive analysis \cite{bishop1967} or Weyl's predicative analysis \cite{weyl1918}. In a series of articles \cite{palmgren1995constructive,palmgren1996constructive,palmgren1998developments}, Erik Palmgren has developed a constructive version of nonstandard analysis which shares many aspects of $\fhyl$. Palmgren's model, however, uses sheaf semantics, and the nonstandard extensions of a set he defines are not themselves sets but rather set-valued functors on a large category of filters. As such, his version of the hyperreal line can be thought of as a dynamic or varying model, a model whose exact domain varies along each object in the category of filters. As we shall see below, Palmgren's construction can be dramatically simplified in a semi-constructive setting, which yields a structure that is essentially equivalent to $\fhyl$.

\subsection{Palmgren's Sheaf Model} \label{pal}

Palmgren's starting point is the observation that, from a constructive perspective, the existence of a nonstandard extension of the real line on which the Transfer and Saturation Principle hold seems to pose some insurmountable problems. Indeed, if $\phi(x)$ is any statement about the natural numbers, then the following is constructively valid:
\[\mathbb{R} \models \forall n \in \mathbb{N} ( \forall m \leq n (m \in \mathbb{N} \to (\phi(m) \jo \neg\phi(m)))).\]
Thus if there is a nonstandard extension $\mathbb{R}^*$ of $\mathbb{R}$ for which the Transfer Principle holds, it follows that:
\[\mathbb{R}^* \models \forall n \in \mathbb{N}^* ( \forall m \leq n (m \in \mathbb{N^*} \to (\phi(m) \jo \neg\phi(m)))),\]
But if $\mathbb{R}^*$ satisfies the Saturation Principle, it contains some infinitely large hypernatural number $N$, which implies that:
\[\mathbb{R}^* \models \forall m \leq N (m \in \mathbb{N}^* \to (\phi(m) \jo \neg\phi(m))).\]
Thus for any natural number $m$, \[\mathbb{R}^* \models \phi(m) \text{ or } \mathbb{R}^* \models \neg\phi(m),\] which by Transfer implies that \[\mathbb{R} \models \phi(m) \text{ or } \mathbb{R} \models \neg \phi(m), \] which is certainly a problematic conclusion for the constructivist, as it would imply that the Law of Excluded Middle holds on the natural number domain. Thus, just like the existence of a Tarskian structure satisfying  the Transfer and Saturation Principle implies in a semi-constructive setting the existence of a non-principal ultrafilter on $\omega$, it also implies in a constructive setting the Law of Excluded Middle on the domain of the natural numbers. Building on some previous work by Moerdijk \cite{Moerdijk1995-MOEAMF}, Palmgren's solution is similar to the solution offered by $\fhyl$ in that the model of constructive nonstandard analysis he develops uses sheaf semantics rather than a constructive version of Tarskian semantics. As a consequence, the nonstandard extension of the reals that Palmgren obtains is a sheaf living in a Grothendieck topos, rather than a Tarskian model in the universe of sets. In what follows, I assume some elementary knowledge of category theory (objects and arrows in a category, functors between categories, natural transformations) and limit myself to a description of Grothendieck toposes and sheaf semantics that is precise enough to give the reader a sense of the relationship between Palmgren's model and $\fhyl$. A detailed exposition of topos theory is beyond the scope of this paper but can be found in \cite{maclane2012sheaves}.\\

Given a category $\cat{C}$, a \textit{sieve} $S$ on an object $A$ in $\cat{C}$ is a collection of arrows in $C$ with codomain $A$ and closed under precomposition, i.e., such that for any arrow $f: B \to A$ in $S$ and any arrow $g: C \to B$, their composition $f\circ g: C \to A$ is in $S$.
A Grothendieck topology $J$ on $\cat{C}$ assigns to every object $A$ in $\cat{C}$ a collection $J(A)$ of \textit{covering sieves} on $A$ in a ``compatible'' way. The pair $(\cat{C},J)$, where $J$ is a Grothendieck topology, is called a \textit{site}. A \textit{presheaf} on $\cat{C}$ is a contravariant functor $\alpha: \cat{C} \to \cat{Sets}$, meaning that $\alpha(A)$ is a set for any object $A$ in $\cat{C}$, and $\alpha(f)$ is a function from $\alpha(B)$ to $\alpha(A)$ whenever $f: A \to B$ is an arrow in $\cat{C}$. Given a site $(\cat{C},J)$, a presheaf $\alpha$ on $\cat{C}$ and a sieve $S = \{f_i : B_i \to A\}_{i \in I}$ covering some object $A$ in $\cat{C}$, a \textit{matching family} for $S$ is a family $\{a_i\}_{i \in I}$ such that $a_i \in \alpha(B_i)$ for every $i \in I$, and for any arrow in $\cat{C}$ $g:D\to B_j$ for some $j \in I$, $\alpha(g)(a_j)=a_k$, where $g \circ f_j = f_k$. Finally, a \textit{sheaf} on a site $(\cat{C},J)$ is a presheaf on $\cat{C}$ such that for any matching family $\{a_i\}_{i\in I}$ for some covering sieve $S = \{f_i : B_i \to A\}_{i \in I}$, there is a unique $a \in \alpha(A)$ such that $\alpha(f_i)(a) = a_i$ for any $f_i \in S$. The category of all sheaves on a site $(\cat{C},J)$ and natural transformations between them, usually written $\mathscr{S}(\cat{C},J)$, is a \textit{Grothendieck topos}.\\

Informally, one can think of a Grothendieck topos as a universe of \textit{varying sets}. Just like the category $\cat{Sets}$ of sets and functions between them, it has enough structure and closure properties to support a significant amount of set-theoretic reasoning. But unlike in $\cat{Sets}$ the objects in a Grothendieck topos $\mathscr{S}(\cat{C},J)$ are sheaves, i.e., contravariant functors from the base category $\cat{C}$ to the category of sets. This means that an object $\alpha$ in a Grothendieck topos may be thought of as a $\cat{C}$-indexed set, meaning that $\alpha$ designates a set $\alpha(A)$ that may vary along the objects in $\cat{C}$. Just like a Tarskian model is a set endowed with an interpretation for the non-logical vocabulary of some first-order language, one can turn sheaves in a Grothendieck topos into models of a first-order language by defining an interpretation for non-logical vocabulary and some recursive clauses for first-order formulas. Because a sheaf $\alpha$ is a varying set, the interpretation and recursive clauses must be defined on every possible value that $\alpha$ can take, and some requirements must be imposed in order to guarantee the validity of certain logical laws. Such a framework is precisely the one provided by sheaf semantics. Formally, given a sheaf $\alpha$ on a site $(\cat{C},J)$ and a first-order language $\la$, one defines first a forcing relation between a set $\alpha(A)$ for some object $A$ in $\cat{C}$, elements of $\alpha(A)$ and atomic $\la$-formulas that must satisfy the following two conditions:

\begin{description}
\item[Monotonicity] For any $n$-ary $\la$-relation symbol $R$, any $n$-tuple $\overline{a}$ of elements of $\alpha(A)$, and any $f:B \to A$, if $\alpha(A) \fr R(\overline{a})$, then $\alpha(B) \fr R(\alpha(f)(\overline{a}))$;
\item[Local Character] For any covering sieve $S = \{f_i: B_i \to A_i\}_{i \in I}$ and any $n$-tuple $\overline{a}$ of elements of $\alpha(A)$, if $\alpha(B_i) \fr R(\alpha(f_i)(\overline{a}))$ for every $i \in I$, then $\alpha(A) \fr R(\overline{a})$.
\end{description}

This forcing relation is then extended to first-order formulas by the following inductive clauses:

\begin{itemize}
    \item If $\phi := \neg \psi$, then $\alpha(A) \fr \phi(\overline{a})$ iff for all $f: B \to A$, $\alpha(B) \nfr \psi(\alpha(f)(\overline{a}))$;
    \item If $\phi := \psi \me \chi$, then $\alpha(A) \fr \phi$ iff $\alpha(A) \fr \psi$ and $\alpha(A) \fr \chi$;
    \item If $\phi := \psi \jo \chi$, then $\alpha(A) \fr \phi(\overline{a})$ iff there is a covering sieve $S = \{f_i : B_i \to A\}_{i \in I}$ of $A$ such that, for any $i \in I$, $\alpha(B_i) \fr \psi(\alpha(f_i)(\overline{a}))$ or $\alpha(B_i) \fr \chi(\alpha(f_i)(\overline{a}))$;
    \item If $\phi := \psi \to \chi$, then $p \fr \phi(\overline{a})$ iff for all $f:B \to A$, $\alpha(B) \fr \psi(\alpha(f)(\overline{a}))$ implies $\alpha(B) \fr \chi(\alpha(f)(\overline{a}))$;
    \item If $\phi := \forall x \psi$, then $A \fr \phi(\overline{a})$ iff for all $f:B \to A$, $\alpha(B) \fr \psi(\alpha(f)(\overline{a}),b)$ for every $b \in \alpha(B)$;
    \item If $\phi := \exists x \psi$, then $A \fr \phi(\overline{a})$ iff there is a covering sieve $S = \{f_i : B_i \to A\}_{i \in I}$ of $A$ such that, for any $i \in I$, $\alpha(B_i) \fr \psi(\alpha(f_i)(\overline{a}),b)$ for some $b \in \alpha(B_i)$.
\end{itemize}

As the reader probably noticed, these clauses are strikingly similar to the inductive clauses of possibility semantics as they were defined in \cref{semdef}. The next section will expand on the relationship between the two semantics, but, for now, we must note an important discrepancy between them. Indeed, these inductive clauses only guarantee that the axioms and rules of first-order intuitionistic logic are valid on any sheaf model, whereas the clauses of possibility semantics guarantee the validity of first-order classical logic on any possibility structure. This is the reason why the ``internal logic'' of Grothendieck toposes, described by its sheaf semantics, is always intuitionistic, but rarely classical, a feature that is particularly appealing to constructivists like Palmgren.\\

We are now in a position to describe Palmgren's constructive model of nonstandard analysis.\footnote{Palmgren's background theory in \cite{palmgren1998developments} is Martin-Löf's constructive type theory, but here I will work in $ZF+DC$ for the sake of simplicity.} This involves first describing the base category $\mathbb{B}$ of filter bases and its Grothendieck topology, and then identifying the sheaves on this site that play the role of the nonstandard extension of standard sets.

\begin{definition}
A \textit{filter base} $\mathscr{F}$ is a pair $(S_F, \{F_i\}_{i \in I})$ such that $S_F$ is a set and $\{F_i\}_{i \in I}$ is a collection of subsets of $S_F$ such that $\{U \sset S_F : F_i \sset U \text{ for some } i \in I\}$ is a filter on $S_F$. A \textit{continuous morphism} between two filter bases $\mathscr{F} = (S_F, \{F_i\}_{i \in I})$ and $\mathscr{G} = (S_G, \{G_j\}_{j \in J})$ is a map $\alpha: S_F \to S_G$ such that for any $j \in J$ there is $i \in I$ such that $F_i \sset \alpha\inv[G_j]$. Filter bases and continuous morphisms form a category $\mathbb{B}$.
\end{definition}

Note that, by its definition, $\mathbb{B}$ is a large category, since any set induces a distinct set of filter bases in $\mathbb{B}$, which implies that the collection of objects in $\mathbb{B}$ is a proper class. It is however easy to see that $\mathbb{B}$ is locally small, meaning that there are only set many arrows betwen any two objects. This implies in particular that the contravariant representable functors $\mathbf{y}_A$, i.e., functors mapping any object $B$ to the collection of all arrows in $\mathbb{B}$ from $B$ to some fixed object $A$, are set-valued and therefore presheaves. Defining sheaves requires the notion of a cover of a filter base:

\begin{definition}
For any filter base $\mathscr{F} = (S_F, \{F_i\}_{i \in I})$, a \textit{cover} of $\mathscr{F}$ is a finite set of maps $\{\alpha_k\}_{k \leq n}$ with $\alpha_k : (S_{G_k}, \{G_j\}_{j \in J_k}) \to \mathscr{F}$ for each $k \leq n$, and such that for any $G_{j_1},...,G_{j_n}$ with $G_{j_k} \in \{G_j\}_{j \in J_k}$ for each $k \leq n$, there is $i \in I$ such that $F_i \sset \bigcup_{k \leq n} \alpha_k[G_{j_k}]$.
\end{definition}

Palmgren goes on to prove that this notion of covering generates a well-behaved Grothendieck topology on $\mathbb{B}$.

\begin{theorem}\label{subc}
The covers on filter bases in $\mathbb{B}$ form a subcanonical Grothendieck topology $\mathscr{J}$ on $\mathbb{B}$. As a consequence, every representable presheaf on $\mathbb{B}$ is a sheaf.
\end{theorem}

Given a set $A$, Palmgren considers the filter base $\mathscr{A} = (A, \{A\})$ and defines the nonstandard extension of $A$ as the representable presheaf $\mathbf{y}_\mathscr{A}$. More concretely, for a filter base $\mathscr{F} = \{S_F, \{F_i\}_{i \in I}\}$, $\mathbf{y}_\mathscr{A}(\mathscr{F})$ is the reduced power of $A$ modulo the filter on $S_F$ generated by $\{F_i\}_{i \in I}$. Since any representable presheaf is a sheaf by \cref{subc}, the Grothendieck topos $\mathscr{Sh}(\mathbb{B},\mathscr{J})$ provides a model for a constructive version of Nelson's Internal Set Theory \cite{nelson1977internal}, in which any standard set has a nonstandard extension, and versions of the Transfer and Saturation Principles hold. Palmgren also proves constructively that analogues of some classical theorems of analysis such as the Heine-Borel Theorem hold in this model. Palmgren's framework has therefore the advantage of being entirely constructive, but quite powerful, which he takes to be an advantage of constructive nonstandard analysis over some other frameworks like Bishop's constructive analysis. From a semi-constructive perspective however, the choice of a Grothendieck topos as a model of nonstandard analysis comes at quite a significant cost, as the logic of such a topos is intuitionistic, but not classical, and sheaf semantics is significantly more involved than Tarskian semantics. One way of overcoming both problems is to modify the base category $\mathbb{B}$ so that presheaves on that category become easier to describe, and the recursive clauses of sheaf semantics become tractable enough to guarantee the validity of classical logic. As we shall see, this can be achieved by choosing $\frec$ as our base category, and this construction will yield that $\fhyl$ is essentially a presheaf on $\frec$, but not a sheaf.

\subsection{Presheaves on the Fréchet Category} \label{sheaf}

The first step in our simplification of Palmgren's approach is to turn the poset $\frec$ into a site. Equivalently, this amounts to \textit{decategorifying} Palmgren's base category $\mathbb{B}$, by replacing it with a poset, which are typically more concrete and tractable objects.

\begin{definition}
The Fréchet Category is the poset category $\frec$, whose objects are all proper filters on $\omega$ extending the Fréchet filter of cofinite subsets, and for any two such filters $F,F'$, there is a unique arrow $F' \to F$ iff $F \sset F'$. Two filters $F, F' \in \frec$ are \textit{incompatible} (noted $F \bot F'$) if there is no $H \in \frec$ such that $H \supseteq F, F'$.
\end{definition}

\begin{definition}
A collection of filters $\{H_i\}_{i \in I}$ \textit{covers} a filter $F \in \frec$ if it satisfies the following three properties:
\begin{itemize}
    \item For any $i \in I$, $H_i \supseteq F$;
    \item For any $i \in I$ and $H \in \frec$, if $H \supseteq H_i$, then $H = H_j$ for some $j \in J$;
    \item For any $G \supseteq F$, there is $i \in I$ such that $H_i \supseteq G$.
\end{itemize}
For any filter $F \in \frec$, we let $\mathscr{C}(F)$ be the collection of all covers on $F$.
\end{definition}

Note that it follows straightforwardly from the definition of a cover that
for any $F \in \frec$, a collection of filters $\{H_i\}_{i \in I}$ covers $F$ if and only if it is a dense open subset of the downset of $F$ in the downset topology on $\frec$. As a consequence, we have the following result:

\begin{lemma}
The covers $\{\mathscr{C}(F)\}_{F \in \frec}$ form a Grothendieck topology on $\frec$.
\end{lemma}

This result is a special case of a well-known result in topos theory, namely that the dense topology on any poset category is a Grothendieck topology \cite[Section III.2]{maclane2012sheaves}.

Let us now fix an $\la$-structure $\mathscr{A}$ with domain $A$.

\begin{definition} The nonstandard extension of $\mathscr{A}$ is the presheaf $\NSA : \frec \to \cat{Sets}$ defined as follows:
\begin{itemize}
    \item Given $F \in \frec$, $\NSA(F)$ is the reduced power of $A$ modulo $F$, that is, the set of equivalence classes of functions $a : \omega \to A$ under the equivalence relation: $$ a \sim_F a' \Leftrightarrow \{i \in \omega: a(i) = a'(i)\} \in F.$$
    \item Given $f: F' \supseteq F \in \frec$, $\NSA(f) : \NSA(F) \to \NSA(F')$ is defined by $\NSA(f)(a_F) = a_{F'}$, where $a_F$ and $a_{F'}$ are the equivalence classes of the function $a : \omega \to A$ under $\sim_F$ and $\sim_{F'}$ respectively.
\end{itemize}
\end{definition}

Note that if $F \sset F'$, then $a \sim_F a'$ implies $a \sim_{F'} a'$, and hence the map $\NSA(f) : \NSA(F) \to \NSA(F')$ is well-defined. In order to turn $\NSA$ into a structure for some language $\la$, we need to define a forcing relation for atomic formulas. By analogy with Palmgren's model, the natural choice here is to use classical satisfaction in the reduced powers $\NSA(F)$ for any $F \in \frec$. Using the notation introduced in \cref{not2}, for any $F \in \frec$, any $n$-ary $\la$-relation symbol $R$ and any $n$-tuple $\overline{a}$, we let $\NSA(F) \models R(\overline{a_F})$ iff $||R(\overline{a})|| \in F$. We can then show that this relation satisfies the monotonicity and local character condition of sheaf semantics:

\begin{lemma}
The forcing relation $\Vdash$ defined as $\NSA(F) \Vdash R(\overline{a_F})$ iff $||R(\overline{a})|| \in F$ for any $F \in \frec$, any $n$-ary $\mathscr{L}$-relation symbol $R$ and any $n$-tuple $\overline{a}$ of elements of $A^\omega$ satisfies the monotonicity and local character conditions of sheaf semantics, namely:
\begin{enumerate}
    \item Monotonicity: For any $f: F' \supseteq F$, $\NSA(F) \Vdash R(\overline{a_F})$ implies $\NSA(F') \Vdash R(\NSA(f)(\overline{a_F}))$;
    \item Local Character: For any cover $\{f_i:F_i \supseteq F\}_{i \in I}$ of some $F \in \frec$, if $\NSA(F_i) \Vdash R(\NSA(f_i)(\overline{a_F}))$ for all $i \in I$, then $\NSA(F) \Vdash R(\overline{a_F})$.
\end{enumerate}

\begin{proof}
First, let us recall that for any morphism $f : F' \supseteq F$ in $\frec$ and any $a \in A^\omega$, $\NSA(f)(a_F) = a_{F'}$. Now we prove that the forcing relation satisfies monotonicity and local character in turn:
\begin{enumerate}
    \item Suppose $f : F' \supseteq F$ and $\NSA(F) \Vdash R(\overline{a_F})$. This means that $||R(\overline{a})|| \in F \sset F'$, so $\NSA(F') \Vdash R(\overline{a_{F'}})$.
    \item Let $\{f_i:F_i \supseteq F\}_{i \in I}$ be a cover of $F$, and suppose that $\NSA(F) \nVdash R(\overline{a_F})$. It is enough to show that there is some $i \in I$ such that $\NSA(F_i) \nVdash R(\overline{a_{F_i}})$. By definition of the forcing relation, $||R(\overline{a})|| \notin F$. Let $H$ be the filter generated by $F \cup -||R(\overline{a})||$. Since $\{f_i\}_{i \in I}$ covers $F$, there is $i \in I$ such that $F_i \supseteq H$. But then it follows that $||R(\overline{a})|| \notin F_i$, so $\NSA(F_i) \nVdash R(\overline{a_{F_i}})$. \qedhere
\end{enumerate}
\end{proof}
\end{lemma}

We can then extend this forcing relation to any first-order formula $\phi$ using the inductive clauses of sheaf semantics. If $\mathscr{A} = \RR$, we therefore obtain a presheaf $\,^*\!\mathcal{R}$ which can be thought of as a ``varying reduced power''. As a matter of fact, we obtain a structure that is equivalent to $\fhyl$ in the following sense:

\begin{theorem}\label{sheafthm}
There is a system of functions ${\pi_F}_{F \in \frec}$ such that for any $F,G \in \frec$:
\begin{itemize}
    \item $\pi_F:\RR^\omega \to \,^*\!\mathcal{R}(F)$;
    \item if $f : G \rset F$, then $\,^*\!\mathcal{R}(f)\circ \pi_F = \pi_G$;
    \item for any $\la$-formula $\phi(x)$ and any tuple $\overline{a}$ of elements of $\RR^\omega$, $\,^*\!\mathcal{R}(F) \fr \phi(\pi_F(\overline{a}))$ iff $\fhyl, F \fr \phi(\overline{a})$.
\end{itemize}
\end{theorem}

\begin{proof}
For any $F \in \frec$, let $\pi_F$ map any $a \in \RR^\omega$ to its equivalence class $a_F$ in the reduced power $^*\!\mathcal{R}(F)$. The rest of the proof is a straightforward induction on the complexity of formulas which exploits the similarity between possibility semantics and sheaf semantics. In particular, in the atomic case, the persistence and refinability conditions correspond to the monotonicity and local character conditions of sheaf semantics respectively, and in the inductive step for disjunctions and existential quantification, one uses the fact that covers on a filter $F$ coincide with dense sets of filters extending $F$.
\end{proof}

 \cref{sheafthm} highlights the similarities between Palmgren's approach and ours. His nonstandard extensions are varying reduced powers of a classical model modulo a proper class of filters in which satisfaction is defined according to sheaf semantics, while we may think of $\fhyl$ as a ``dynamic'' reduced power of $\RR$ modulo the non-principal filters on $\omega$, in which satisfaction is defined locally and amalgamated in a coherent way by possibility semantics. However, taking the Grothendieck topos of sheaves over the site $(\frec, \{\mathscr{C}(F)\}_{F \in \frec})$ does not yield a nonstandard universe of sets, as many presheaves, including the presheaf corresponding to $\fhyl$, fail to be sheaves.

\begin{fact}
Let $\mathscr{A}$ be an $\mathscr{L}$-structure with a domain $A$ such that there is a injection $\pi:2^\omega \to A$. Then $\NSA$ is not a sheaf.
\end{fact}

\begin{proof}
Recall that for any $F \in \frec$ and any cover $\{F_i\}_{i \in I}$ of $F$, a \textit{matching family} is a family $\{a_i\}_{i \in I}$ such that $a_i \in \NSA(F_i)$ for any $i \in I$, and for any $i, j \in I$ with $f: F_j \supseteq F_i$, $\NSA(f)(a_i) = a_j$. To show that $\NSA$ is not a sheaf, it is enough to find a cover $\{F_i\}_{i \in I}$ of a some $F \in \frec$ and a matching family $\{a_i\}_{i \in I}$ such that for all $a \in A^\omega$ there is $i \in I$ with $a \notin a_i$.\\ Let $F$ be the Fréchet filter, and recall that one can prove in $ZF+DC$ that there is an almost disjoint family $\{A_f\}_{f \in 2^\omega}$ of infinite subsets of $\omega$ indexed by functions from $\omega$ to $\{0,1\}$ \cite[p.~118]{jech2013set}. For any $f: \omega \to 2$, let $F_f$ be the filter generated by $F \cup A_f$. Note that, whenever $f \neq g$, $F_f$ and $F_g$ have no common extension since $|A_f \cap A_g| < \aleph_0$. Let $\{H_j\}_{j \in J}$ be the collection of all $H \supseteq F$ such that $H\bot F_f$ for every $f \in 2^\omega$, and let $D = \bigcup_{f \in 2^\omega} \{H \in \frec : H \supseteq F_f\} \cup \{H_j\}_{j \in J}$. By construction $D$ is a cover of $F$. Now for any $f \in 2^\omega$, let $a^f:\omega \to A$ be the constant function with range $\{\pi(f)\}$, and let $\alpha: \omega \to A$ be such that $\alpha(i) = \pi(\chi_{\{i\}})$ for any $i \in \omega$, where $\chi_{\{i\}}$ is the characteristic function of the singleton set $\{i\}$. Now consider the family $C = \bigcup_{f \in 2^\omega}\{a^f_H : H \supseteq F_f\} \cup \{\alpha_{H_j}\}_{j \in J}$. Clearly, since on the one hand we selected the $H_j$-equivalence class of $\alpha$ for any $j \in J$ and the $H$-equivalence class of $a^f$ for any $f \in 2^\omega$ and any $H$ extending $F_f$, and since on the other hand for any $f \in 2^\omega$, $F_f$ is incompatible with $H_j$ for any $j \in J$ and with any $H\supseteq F_g$ for any $g \neq f \in 2^\omega$, $C$ is a matching family. However, for any $a : \omega \to A$ and any distinct $f, g : \omega \to 2$, observe that $||a = a^f|| \cap ||a = a^g|| = \emptyset$. Moreover, since $\omega  \supseteq \bigcup_{f \in 2^\omega} ||a = a^f||$, it must be the case that $||a = a^f|| = \emptyset$ for some $f \in 2^\omega$ (otherwise, one could define a injection from $2^\omega$ into $\omega$, contradicting Cantor's theorem). Hence for any $a: \omega \to A$ there is $f \in 2^\omega$ such that $a \notin a^f_{F_f}$, which shows that $\NSA$ is not a sheaf.
\end{proof}

Let us note that, if we assume the Axiom of Choice, the argument above can be easily be modified to show that $\NSA$ is not a sheaf for any uncountable structure $\mathscr{A}$. Whether this is a necessary condition on $\mathscr{A}$ for $\NSA$ not to be a sheaf is left as an open problem. In any case, since we can define in $ZF+DC$ an injection from $2^\omega$ into $\RR$, it follows that $\,^*\!\mathcal{R}$ is not a sheaf. Since the nonstandard extensions we defined are not sheaves in general, we cannot work in the Grothendieck topos of sheaves over $\frec$ and straightforwardly apply the machinery of sheaf semantics. Fortunately, possibility semantics still allows us to have a notion of a model which, as we have seen in the first section, is robust enough to develop a significant part of nonstandard analysis \textit{à la} Robinson. In other words, the ``decategorification'' of the category of filters in Palmgren's approach, while allowing for a drastic simplification of many technical details, comes at the cost of not providing a nonstandard ``ambient universe'', like the one postulated in Nelson's Internal Set Theory or offered by the Grothendieck topos of sheaves over $\mathbb{B}$. In that sense, $\fhyl$ can be seen as a watered down version of Palmgren's constructive nonstandard analysis: because we work in a semi-constructive setting, we do not need the full generality of sheaf semantics nor the full power of topos theory, and we can work instead in a possibility structure which resembles more closely Tarskian semantics. Additionally, the fact that $\fhyl$ can be presented as a presheaf on a poset category endowed with the dense topology is, in itself, not surprising in the slightest. Indeed, sheaf semantics was originally conceived as a generalization of Cohen forcing, which corresponds precisely in topos theory to working in a topos of sheaves over a poset category endowed with the dense topology \cite[Section 6.2]{maclane2012sheaves}. As we will now see, $\fhyl$ itself is tightly connected to forcing and to Boolean-valued models of analysis.

\section{The Generic Approach} \label{gen}

The third and last alternative approach to classical nonstandard analysis that I consider here involves using Boolean-valued models rather than two-valued models. Unlike the previous two, it does not stem from a desire to make nonstandard analysis more constructive but instead from an interest in viewing the ultrapower construction as a special case of a more general kind of algebraic construction. This approach is also the one with the tightest relationship to set-theoretic forcing, and, as we will see, will motivate our last perspective on $\fhyl$ as a generic approximation of a Luxemburg ultrapower.

\subsection{Boolean-Valued Models} \label{scott}
 
In \cite{scottboolean}, Scott remarks that Boolean-valued models, even though they appeared originally in the context of forcing in set theory, can also be seen as a generalization of ultraproducts:

\begin{quote}
    The idea of constructing Boolean-valued models could have been (but was not) discovered as a generalization of the ultraproduct method used now so often to obtain nonstandard models for ordinary analysis. Roughly, we can say that ultraproducts use the \textit{standard} Boolean algebras (the power-set algebras) to obtain models elementarily equivalent to the standard model, whereas the Boolean methods allows the nonstandard complete algebras (such as the \textit{Lebesgue} algebra of measurable sets modulo sets of measure zero or the \textit{Baire} algebra of Borel sets modulo sets of the first category). Thus the Boolean method leads to \textit{nonstandard} nonstandard models that are not only not isomorphic to the standard model but are not even equivalent. Nevertheless, they do satisfy all the usual axioms and deserve to be called models of analysis. \cite[pp.~87-88]{scottboolean}
\end{quote}

Scott goes on to explain how ultraproducts can be seen as a special case of quotients of a Boolean-valued model. Given a family of Tarskian $\la$-structures $\{\spa{A}_i\}_{i \in I}$, one may consider the Boolean-valued model whose domain is the direct product $A = \Pi_{i \in I} A_i$ of the domains of the models $\{\spa{A}_i\}_{i \in I}$, the Boolean algebra of truth-values is the powerset of the index set $\Po(I)$, and for any tuple $\overline{a}$ of elements of $A$, and any $\la$-formula $\phi(\overline{x})$, the truth-value of $\phi(\overline{a})$ is $||\phi(\overline{a})|| = \{i \in \omega : \spa{A}_i \models \phi(\overline{a}(i))\}$. A straightforward induction on the complexity of formulas shows that this indeed correctly defines a Boolean-valued model. An ultraproduct of the models $\{\spa{A}_i\}_{i \in I}$ can then be obtained by taking an ultrafilter on $\Po(I)$, which is the same as a Boolean homomorphism $f_U : \Po(I) \to \{0,1\}$, quotienting the domain $A$ by the equivalence relation $a \sim_U b$ iff $f_U(||a = b||) =1$, and defining a Tarskian model on this domain by letting $\phi(\overline{a_U})$ be true in the model iff $f_U(||\phi(\overline{a})||) = 1$ for $\phi(\overline{x})$ a $\la$-formula and $\overline{a_U}$ the tuple of equivalence classses of a tuple $\overline{a}$ of elements of $A$. By \L\"os's Theorem, the model thus obtained is precisely the ultraproduct of the family of models $\{\spa{A}_i\}_{i \in I}$ by the ultrafilter $U$. As Scott concludes:

\begin{quote}
    In short, we have divided the ultraproduct construction into two stages: \textit{product} followed by \textit{ultra}. It is the generalization of the product part we wish to emphasize. [p.~89]
\end{quote}

Scott goes on to show how to define a $B$-valued model of analysis in a similar fashion when $B$ is a complete Boolean algebra, but not necessarily the powerset set of some set $I$, and that one may still retrieve Tarskian models as quotients of such Boolean-valued models modulo an ultrafilter on $B$. In other words, the Boolean-valued models of analysis so constructed are ``one ultrafilter away'' from being Tarskian, nonstandard models of analysis.\\

Let us now briefly see how $\fhyl$ connects to Boolean-valued models. As mentioned in \cref{poss}, given an $\la$-possibility structure $(\mathfrak{P},D,\itp)$, an $\la$-formula $\phi(\overline{a})$ and a tuple $\overline{a}$ of elements of $D$, the set $\llbracket \phi(\overline{a}) \rrbracket = \{p \in P : p \fr \phi(\overline{a})\}$ is always a regular-open subset of $\mathfrak{P}$ and the regular opens of any poset $\mathfrak{P}$ always form a complete Boolean algebra $\RO(\mathfrak{P})$. We may therefore think of any possibility structure $(\mathfrak{P},D,\itp)$ as a Boolean-valued model with domain $D$ and algebra of truth values $\RO(\mathfrak{P})$. Since any complete Boolean algebra is isomorphic to $\RO(\mathfrak{P})$ for some poset $\mathfrak{P}$, this means that possibility semantics is as general as the Boolean-valued models discussed by Scott.

What about the algebra of truth-values of $\fhyl$ specifically? The non-principal filters on $\Po(\omega)$ are in one-to-one correspondence with the proper filters on $\Po(\omega)^*$, the Boolean algebra obtained by quotienting $\Po(\omega)$ by the ideal of finite sets. Moreover, it follows from \cite[Theorem 5.49]{Hol16} that the algebra of regular open sets of the poset of all proper filters on a Boolean algebra $B$ is isomorphic to the \textit{canonical extension} of $B$, usually written $B^\delta$.\footnote{See \cite{Dunn} for more on canonical extensions.} Thus $\fhyl$ can be thought of as a $(\Po(\omega)^*)^\delta$-valued Boolean valued model of analysis. Since the Ultrafilter Lemma is equivalent to the statement that the canonical extension of any Boolean algebra $B$ is always isomorphic to the powerset of some set (in fact, of the set $X_B$ of all ultrafilters on $B$), it holds for $\Po(\omega)$ if and only if $\fhyl$ is a ``standard'' nonstandard Boolean-valued model of analysis, to use Scott's terminology. As we will see in the next section, there is indeed a strong sense in which $\fhyl$ is ``one ultrafilter away'' from a Luxemburg ultrapower.

\subsection{Generic Models and Luxemburg ultrapowers} \label{ult}
As the comparison with Scott's Boolean-valued models of analysis in the previous section showed, we may think of $\fhyl$ as a Boolean-valued model that captures the semi-constructive part of the standard ultrapower construction, i.e., the ``product'' part. If we think of a Boolean-valued model as a ``fuzzy'' Tarskian model, in which the classical truth-value of every formula is not always settled, we may therefore think that $\fhyl$ is a partial approximation of a Luxemburg ultrapower, the best thing we can get in the absence of the Axiom of Choice. Similarly, viewpoints in $\frec$ can also be interpreted as approximations of classical Luxemburg ultrapowers in a fairly strong sense. A simple way to flesh out the details of the idea is to exploit the fact that possibility semantics coincides with forcing semantics. As such, it is natural to import some elementary concepts from forcing into the study of possibility models and in particular the notion of a generic filter on a forcing poset.

\begin{definition}
A subset $D$ of $\frec$ is \textit{dense} if for any $F \in \frec$ there is $F' \in D$ such that $F \sset F'$. A \textit{generic filter} on $\frec$ is a directed subset $G$ of $\frec$ (meaning that for any $F,F' \in G$ there is $H \in G$ such that $H \supseteq G,G'$) that is upward closed (meaning that $F \sset F' \in G$ implies $F \in G$) and has non-empty intersection with every dense subset of $\frec$.
\end{definition}

A generic filter on a forcing poset $\mathbb{P}$ can usually be thought as a coherent way of choosing a maximal set of conditions in that poset. Since we want to think of viewpoints in $\frec$ as partial approximations of a Tarskian model, it is then natural to wonder whether a generic filter on $\frec$ would allow us to define such a Tarskian model. The following lemma shows that this is indeed possible.

\begin{lemma}
Let $G$ be a generic filter on $\frec$. Then there is an equivalence relation $\sim_G$ on $\RR^\omega$ and a Tarskian model $\RR_G$, with domain $\{a_G : a \in \RR^\omega\}$, such that for any tuple $\overline{a} \in \RR^\omega$ and any formula $\phi$, $\RR_G \models \phi(\overline{a_G})$ iff there is $F \in G$ such that $F \Vdash \phi(\overline{a})$.
\end{lemma}

\begin{proof}
Let $G$ be generic, and for any $a, a' \in A^\omega$, let $a \sim_G a'$ iff there is $F \in G$ such that $a_F = a'_F$. It is routine to verify that $\sim_G$ is an equivalence relation (transitivity relies on $G$ being directed). Let the domain of $\RR_G$ be $\{a_G: a \in \RR^\omega\}$, and for any relation symbol $R$, let $\overline{a_G} \in \mathscr{I}(R)$ iff there is $F \in G$ such that $F \Vdash R(\overline{a})$. Once again, directedness of $G$ ensures that this is well-defined. Let $\mathscr{A}(G)$ be the Tarskian model thus defined. We claim the following:
\begin{align*}
    \forall \overline{a_G}: \mathscr{\RR}_G \models \phi(\overline{a_G}) \Leftrightarrow \exists F \in G: F \Vdash \phi(\overline{a})
\end{align*}

The proof is a straightforward induction on the complexity of $\phi$. Note however that the right-to-left direction of the existential case holds even in a semi-constructive setting, because of the fullness of $\mathscr{A}$ mentioned in \Cref{rmk1}.
\end{proof}

Generic models, determined by generic filters on $\frec$, are thus the classical Tarskian models that are being approximated by viewpoints in $\frec$. As we shall see below, they are also tightly connected to Luxemburg ultrapowers. The next lemma establishes a one-to-one correspondence between generic filters on $\frec$ and non-principal ultrafilters on $\omega$, which can then be used to show that every generic model is isomorphic to a Luxemburg ultrapower, and vice-versa. 

\begin{lemma} \label{corr}
The generic filters on $\frec$ are in one-to-one correspondence with the non-principal ultrafilters on $\omega$.
\end{lemma}

\begin{proof}
Given $G$ a generic filter on $\frec$, let $\alpha(G) = \bigcup G$. To see that this is well defined, observe first that, since $G$ is a filter on $\mathscr{G}$, it is a directed family of filters on $\omega$, and thus its union is also a filter on $\omega$. More over, for any $A \sset \omega$, the set $D_A = \{F \in \frec : A \in F \text{ or } -A \in F\}$ is dense. Since $G$ is generic, this means that $D_A \cap G \neq \emptyset$, hence $A \in \bigcup G$ or $- A \in \bigcup G$. This shows that $\alpha(G)$ is an ultrafilter on $\omega$, and since $\alpha(G)$ contains the Fréchet filter, it is clearly non-principal.

Conversely, given a non-principal ultrafilter $U$ on $\omega$, let $\beta(U) = \{F \in \frec : F \sset U\}$. Since $\beta(U)$ is the principal upset generated by $U$ in $\frec$, it is clearly a filter. Moreover, if $U$ is an ultrafilter, then it is an atom in $\frec$, meaning that $F' = U$ for any $F' \supseteq U$. Hence $U \in D$ for any dense subset $D$ of $\frec$, from which it follows that $\beta(U)$ is generic.

Next, we check that for any generic $G$, $\beta(\alpha(G)) = \{F \in \frec : F \sset \bigcup G\} = G$. The right-to-left inclusion is immediate. For the converse, suppose $F \notin G$, and consider the set $D_F = \{H \in \frec : H \supseteq F \text{ or } H\bot F\}$. Clearly, $D_F$ is dense, so there is $H \in D_F \cap G$. Moreover, since $F \notin G$, it must be the case that $F \bot H$. Hence there is $B \in H$ such that $\neg B \in F$. Since $H \sset \bigcup G$ and $\bigcup G$ is a filter on $\omega$, this means that $F \nsubseteq \bigcup G$.

Similarly, we claim that for any ultrafilter $U$, $\alpha(\beta(U)) = \bigcup \{F \in \frec : F \sset U\} = U$. The left-to-right inclusion is clear, and the converse follow from the fact that $\alpha(\beta(U))$ is an ultrafilter, hence maximal.
\end{proof}

Let us note that, assuming the Ultrafilter Lemma, there is also an elegant proof of the result above that uses the theory of Boolean algebras. By the observation made at the end of \cref{scott}, the algebra $\RO(\frec)$ of regular open subsets of $\frec$ is isomorphic to the canonical extension of $\Po(\omega)^*$, which, by Stone duality, is itself isomorphic to $\Po(X_{\Po(\omega)^*})$, the powerset of the set of ultrafilters on $\Po(\omega)^*$. Moreover, it is a well-known result in the forcing literature \cite[p.~156]{jech1997set} that a generic filter on a poset $\mathbb{P}$ corresponds to a complete non-principal ultrafilter on $\RO(\mathbb{P})$, where an ultrafilter $U$ on a complete Boolean algebra $B$ is complete if $\bigme S \in U$ whenever $S \sset U$, for any $S \sset B$. This means that any generic filter on $\frec$ is essentially a principal ultrafilter on $\Po(X_{\Po(\omega)^*})$, i.e., it is induced by a point in $X_{\Po(\omega)^*}$. But points in $\Po(X_{\Po(\omega)^*})$ are ultrafilters on $\Po(\omega)^*$, which are in one-to-one correspondence with non-principal ultrafilters on $\omega$. However, the direct construction of the one-to-one correspondence between generic filters on $\frec$ and non-principal ultrafilters on $\omega$ described in the proof of \cref{corr} also allows us to establish an equivalence between generic models and Luxemburg ultrapowers.

\begin{theorem} \label{geniso}
For any non-principal ultrafilter $U$ on $\omega$, the ultrapower induced by $U$ is isomorphic to the generic model determined by $\beta(U)$. Conversely, for any generic filter $G$ over $\frec$, the generic model $\RR_G$ is isomorphic to the ultrapower induced by $\alpha(G)$.
\end{theorem}

\begin{proof}
Note that for any non-principal ultrafilter $U$ and any $a, a' \in \RR^\omega$, $$ a \sim_U a' \text{ iff } ||a = a'|| \in U \text{ iff } \exists F \in \beta(U) : ||a = a'|| \in F \text{ iff } a \sim_{\beta(U)} a'.$$ Thus the map $a_U \mapsto a_{\beta(U)}$ is well-defined and is easily seen to be an isomorphism between $\RR_U$ and $\RR_{\beta(U)}$. Similarly, for any generic filter $G$, the map $a_G \mapsto a_{\alpha(G)}$ is well-defined, and is an isomorphism between $\RR_G$ and $\RR_{\alpha(G)}$.
\end{proof}

As a consequence of \cref{geniso}, we may now think of points in $\frec$ as partial approximations of a Luxemburg ultrapower in a very precise way. It is easy to see that, assuming the Ultrafilter Lemma for $\Po(\omega)^*$, or, equivalently, the existence of generic filters over $\frec$ containing any $F \in \frec$, we have that, for any formula $\phi(\overline{x})$, any tuple $\overline{a} \in \RR^\omega$ and any $F \in \frec$, $F \fr \phi(\overline{a})$ if and only if $\RR_U \models \phi(\overline{a_U})$ for any ultrafilter $U$ that extends $F$. Indeed, if $F \sset U$ for some ultrafilter $U$, then $F \in \beta(U)$, hence $F \fr \phi(\overline{a})$ implies $\RR_{\beta(U)} \models \phi(\overline{a_U})$. Conversely, if $F \not \fr \phi(\overline{a_U})$, then there is $F' \rset F$ such that $G \fr \neg \phi(\overline{a})$, and hence there is a generic filter $G$ such that $F,F' \in G$ for which $\RR_G \models \neg \phi(\overline{a_G})$. But then $\alpha(G)$ is an ultrafilter such that $F \sset \alpha(G)$ and $\RR_{\alpha(G)} \not \models \phi(\overline{a_{\alpha(G)}})$. Thus satisfaction of a formula $\phi$ at a viewpoint $F$ corresponds to truth in all ultrapowers induced by an ultrafilter extending $F$. In the presence of a strong enough fragment of the Axiom of Choice like the Ultrafilter Lemma, this means that one can take any $F \in \frec$ for a representative of the set of all Luxemburg ultrapowers induced by an ultrafilter extending $F$. As satisfaction at $F$ coincides with truth in all ultrapowers induced by an ultrafilter extending $F$, this means that any filter $F$ can be thought of as providing a partial viewpoint on $\fhyl$, in the sense that the satisfaction relation on $F$ captures exactly how much information we would have about a classical hyperreal line if we knew only that $F$ is in the ultrafilter. This also means that $F_0$, the Fréchet filter, truly stands for a ``generic'' ultrapower, since $F_0$ forces precisely those formulas that would be satisfied by any Luxemburg ultrapower.

\section{F-Hyperreals and Objections to the Hyperreals} \label{obj}
In the previous sections, I have introduced $\fhyl$ as a possibility structure (\cref{PossPow}), explored some of its mathematical properties (\cref{fhyl}), and argued that it arises naturally as a point of convergence for several distinct mathematical endeavors (\cref{asym,dyn,gen}) Along the way, I have also compared $\fhyl$ with these alternative approaches to nonstandard analysis and argued that it has several technical advantages over its competitors, as it is a more powerful version of Laugwitz's and Tao's approach with reduced powers, a simpler framework than Palmgren's sheaf-theoretic approach, and a more concrete version of the Boolean-valued approach of Scott. In this final section, I will relate the $F$-hyperreals to the classical Robinsonian hyperreal line from a philosophical and methodological perspective. Robinson's work and its significance for the role that infinitesimals can play both in mathematics and in the empirical sciences has sparked many debates over the years. Here, I want to focus on objections that have been raised against the applicability of nonstandard analysis to mathematics itself, rather than the broader debate of its applicability to other sciences. In particular, I will examine in turn two distinct worries about Robinsonian hyperreals, namely, that they introduce impure methods in analysis and that they do not form a canonical structure. In both cases, I will argue that our semi-constructive $F$-hyperreal line fares at least as well as, and often better than, the classical hyperreal line with respect to these arguments.

\subsection{The argument from purity and definability}\label{pur}

There is a longstanding tradition of criticisms of Robinsonian nonstandard analysis. As we have seen in \cref{tao}, Tao expressed some uneasiness towards the fact that classical nonstandard analysis uses infinitary, highly non-constructive objects like non-principal ultrafilters to derive results about ordinary mathematical objects such as sets of real numbers. Dauben \cite{dauben1987abraham} recalls that Bishop, who proposed an alternative, constructive approach to analysis, criticized the use of notions from nonstandard analysis to teach elementary calculus in harsh terms, claiming that ``[i]t is difficult to believe that debasement of meaning could be carried so far'' \cite[pp.~513-514]{bishop1975}. Dauben also argues that Robinson's interest in the history of infinitesimal calculus was, at least in part, motivated by the desire to show that the tools and methods of nonstandard analysis provided in fact a natural conceptual framework for standard analysis:

\begin{quote}
    History could serve the mathematician as propaganda. Robinson was apparently concerned that many mathematicians were prepared to adopt a ``so what'' attitude toward nonstandard analysis because of the more familiar reduction that was always possible to classical foundations [...]. But, as Robinson also began to argue with increasing frequency and in greater detail, \textit{historically} the concept of infinitesimals has always seemed natural and intuitively preferable to more convoluted and less intuitive sorts of rigor. \cite[p.~184]{dauben1987abraham}
\end{quote}

It seems that the debate between Robinson and his critics here can be conveniently phrased in terms of purity of methods, in particular in terms of \textit{topical} and \textit{elemental} purity. As discussed in detail in \cite{arana2017alleged,arana2011purity}, topically pure proofs of a theorem are proofs whose content does not exceed the content of the theorem, i.e., proofs that do not involve objects or concepts that are foreign to the topic of the theorem. By contrast, elementally pure proofs of a theorem are proofs whose proof-theoretic or foundational resources do not exceed those of the theorem. Both kinds of concerns seem to be routinely raised against the methods of nonstandard analysis. Bishop's criticism that the use of nonstandard concepts in elementary calculus is a ``debasement of meaning'' can be understood as the complaint that such proofs disregard the meaning of elementary notions of calculus like that of a limit, while Tao's uneasiness with non-principal ultrafilters and the infinitary aspect of nonstandard methods speaks to a desire for elementally pure proofs. 

How does $\fhyl$ compare to the classical hyperreal line in that respect? I think we may argue that it performs as least as well as Robinson's hyperreal line from the topical point of view, i.e., regarding the continuity with historical and contemporary mathematical practice, and strictly better from the foundational or elemental viewpoint. Indeed, the asymptotic perspective on $\fhyl$ presented in \cref{asym} highlights its continuity with both historical developments in infinitesimal calculus, as evidenced by the fact that $\fhyl$ is a strengthening of Laugwitz's approach with reduced powers, and with contemporary mathematical practice, since it can be seen as a formal counterpart to Tao's cheap nonstandard analysis. As a simplified version of Palmgren's approach, it might even seem more appealing to constructivists than the classical hyperreal line. From the foundational perspective on the other hand, I believe that $\fhyl$ is clearly closer to the usual conceptual resources of analysis, as the construction can be carried out in a semi-constructive setting. One might think that the reliance on \textit{some} amount of choice puts $\fhyl$ in a similar situation as the classical hyperreal line, but I think that argument can be resisted. Indeed, many central results from nineteenth-century analysis rely heavily on the Axiom of Dependent Choices, while the original reception of Zermelo's full Axiom of Choice among analysts was lukewarm as best \cite[Chap.~23]{ewald1996kant}. As Moschovakis writes in \cite{moschovakis2005notes}:

\begin{quote}
    We have remarked that before it was formulated precisely by Zermelo, the Axiom of Choice had been used many times ``silently” in classical mathematics, and in particular in analysis. \textit{These classical applications, however, can all be justified on the basis of the Axiom of Dependent Choices}---in fact most of them need only the weaker Countable Principle of Choice. [...] This difference between the choice principles needed for classical mathematics and those required by Cantor’s new theory of sets explains in part the strident reaction to the axioms of Zermelo by the distinguished analysts of his time (including the great Borel), who had used choice principles routinely in their work—and continued using them, as they denounced general set theory and called it an illusion: in the context of 19th century classical analysis, the Axiom of Dependent Choices is natural and necessary, while the full Axiom of Choice is unnecessary and even has some counterintuitive consequences, including certainly the Wellordering Theorem. \cite[pp.~116-117]{moschovakis2005notes}
\end{quote}

We may therefore argue that the proof-theoretic resources required by $\fhyl$ fit squarely within the natural resources required by classical analysis. Unlike in Robinsonian nonstandard analysis, the fragment of the Axiom of Choice that belongs to the semi-constructive setting does not yield paradoxical consequences from the point of view of the theory of the real numbers, like a well-ordering of the reals or the existence of a non-Lebesgue measurable set. In fact, this latter point also allows us to draw a sharp contrast between $\fhyl$ and the classical hyperreal line with respect to an influential criticism of nonstandard analysis recently voiced by Alain Connes.\\

In a famous passage in the first chapter of \cite{connes}, Connes, discussing the relationship between ``logic and reality'', raises the following objection to the use of nonstandard analysis in ordinary mathematics:

\begin{quote}
    But in the final analysis, I became aware of an absolutely major flaw in this theory, an irremediable defect. It is this: in nonstandard analysis, one is supposed to manipulate infinitesimals; yet, if such an infinitesimal is given, starting from any given nonstandard number, a subset of the interval $[0,1]$ automatically arises which is not measurable in the sense of Lebesgue. [...] The conclusion I drew was that no one will ever be able to show me a nonstandard number. \cite[p.~16]{connes}
\end{quote}

This leads him to argue that nonstandard analysis does not describe the ``primordial mathematical reality'' that he claims is the true subject of mathematics:

\begin{quote}
    What conclusion can we draw about nonstandard analysis? This means that, since no one will ever be able to name a nonstandard number, the theory remains virtual [...]. [p.~17]
\end{quote}

Connes doesn't specify exactly what he means by ``naming'' or ``showing'' a non-standard number, but his argument seems to go along the following lines. Non-measurable sets of real numbers are highly abstract objects, which tend not to appear in ordinary mathematical practice. Connes cites here Solovay's model of $ZF+DC \text{ }+$ ``All sets of reals are Lebesgue measurable'' in \cite{solovaymodel} as evidence that no non-measurable set of reals would ever be encountered in ordinary mathematics. However, since one may use a non-principal ultrafilter on $\omega$ to define a non-measurable set of reals and, as we have seen before, such ultrafilters can themselves be defined from hyperfinite natural numbers in a sufficiently saturated elementary extension of the reals, such hyperfinite natural numbers must be as abstract and ineffable as nonmeasurable sets of reals. Although Connes does not explicitly phrase his argument in terms of purity considerations, it seems nonetheless possible to understand his worry as such and stemming from similar considerations as Schechter's distinction between quasi-constructive mathematics and \textit{intangibles}. The crucial problem of nonstandard analysis is that it uses resources beyond the universe of ordinary mathematics, ``virtual'' mathematical objects that cannot be ``named'' or explicitly constructed, in order to derive results about what Connes calls the primordial mathematical reality.\\

As noted by Kanovei et al. in \cite{kanovei2013tools}, the issue of whether hyperreal numbers can be ``named'' or ``defined'' is a delicate problem. One the one hand, Solovay showed that it is consistent with $ZFC$ that no nonmeasurable set of the reals is ordinal definable (as the ordinal definable sets in this model form precisely Solovay's model of $ZF+DC+$ ``All sets of reals are Lebesgue measurable''), which yields a model in which no non-principal ultrafilter on $\omega$ is ordinal definable. On the other hand, G\"odel's constructible universe $L$ provides a well-known example of a model of $ZFC$ in which nonmeasurable sets appear very low in the projective hierarchy and are not only ordinal definable but in fact constructible. Kanovei et al. also cite a result in \cite{kanovei2004definable} in which the existence of an ordinal definable hyperreal extension is proved in $ZFC$.

I do not wish to adjudicate here the debate between Connes and his critics regarding whether Robinsonian hyperreals can be named or not, but I would rather like to point out that the debate does not arise in the case of the $F$-hyperreals. Indeed, since $\fhyl$ is constructed assuming only $ZF+DC$, it can be constructed even in Solovay's model. Moreover, elements in $\fhyl$ can easily be named and are very familiar objects, since they are merely countable sequences of real numbers. Exhibiting an $F$-infinitesimal for any $F \in \frec$ is exceedingly easy, since any $F_0$-infinitesimal is an $F$-infinitesimal for any $F \in \frec$, and such infinitesimals are easily defined. Of course, determining \textit{whether} a given element $a \in \fhyl$ is an $F$-infinitesimal for some $F\in \frec$ is a much more complex task, just like determining whether two elements $a,b \in \frec$ are identified at $F$. One might be tempted to conclude from this that not much is truly gained by moving away from a Luxemburg ultrapower to $\fhyl$. Indeed, one might argue that one could slightly modify the ultrapower construction to take countable sequences of reals as elements, and define the semantics of the equality predicate as the equivalence relation on this domain induced by the ultrafilter, rather than as \textit{bona fide} equality. Clearly, in such a construction, one could ``name'' infinitesimals just as easily as in $\fhyl$. But there is a significant difference between the two constructions that I think is worth emphasizing. As noted in \cref{fhyl}, validity in $\fhyl$ coincides with the forcing relation at $F_0$, which is itself \textit{absolute} from the existence of ultrafilters in a rather strong sense. By \cref{losx}, determining whether a first-order sentence is forced at $F_0$ only requires examining whether the set of natural numbers it determines is cofinite and does not depend on the existence of any non-principal ultrafilter on $\omega$. By contrast, even with the modification proposed above, Luxemburg ultrapowers still require the existence of non-principal ultrafilters to be defined, and \L\'os's Theorem shows that satisfaction in such a model is truly determined by the ultrafilter. Thus the real gain from taking $\fhyl$ as our hyperreal line comes from the fact that satisfaction in the model becomes more tractable, without losing the power of the Transfer and Saturation Principles. 

Therefore, in contrast with the classical hyperreal line, $\fhyl$ seems to be an acceptable structure even from Connes's viewpoint. In fact, the generic perspective on $\fhyl$ seems to align quite well with Connes's claim that the theory of Luxemburg ultrapowers ``remains virtual'', since it establishes in a semi-constructive setting that Luxemburg ultrapowers are essentially obtained by forcing over $\frec$. At the same time, it allows us to disentangle a theory of hyperreals from the theory of ultrapowers  and to identify only the latter as ``virtual'', because of its reliance on the Ultrafilter Lemma.
To sum up, I have argued that, regarding criticisms against nonstandard analysis that are rooted in purity concerns broadly understood, the semi-constructive approach fares better than the classical approach. The reasons to argue for a topical continuity between nonstandard methods and historical and contemporary mathematical practice remain valid, while from the foundational or elemental viewpoint, the fact that the Axiom of Dependent Choices is enough for the semantics of $\fhyl$ to behave in a tractable way is a clear advantage over the classical hyperreal line.

\subsection{The argument from canonicity}\label{canon}

I will now discuss a second kind of argument against the nonstandard approach, which points out a certain kind of arbitrariness of the hyperreal line. By contrast with the first kind of argument, which targets the methods of nonstandard analysis and in particular its reliance on the Axiom of Choice, this second line of argument targets the very structure of the hyperreal line, arguing that it lacks the kind of mathematical properties enjoyed by other standard number systems. An eloquent proponent of this view is Machover in \cite{machover1993place}. Reflecting on the developments in nonstandard analysis in the first thirty years after Robinson's seminal work, Machover observes that the mathematical community at large has not embraced nonstandard methods as quickly as some had hoped and offers a explanation for this phenomenon. According to him, a central problem with nonstandard analysis lies in the fact that the structures it studies, enlargements of the standard universe, fail to be \textit{canonical} in the way that the natural numbers or the reals are:

\begin{quote}
    The point is that whereas the classical number systems (the integers, the rationals the reals etc.) are canonical, there is no such thing as \textit{the} canonical system of $^*$integers, $^*$rationals or $^*$reals. The former can be characterized (informally or within set theory) uniquely up to isomorphism by virtue of their \textit{mathematical} properties: for example, the field of rationals is the smallest field containing the integers, and the field of reals is the completion of the field of rationals. But there is no such thing as \textit{the} [enlarged] field of $^*$reals. The $^*$reals (and in particular the infinitesimal $^*$reals) we happen to deal with in a given nonstandard discourse depend on the enlargement chosen. There is no known way of singling out a particular enlargement that can plausibly be regarded as canonical, nor is there any reason to be sure that a method for obtaining a canonical enlargement will necessarily be invented. \cite[pp.~207-208]{machover1993place} 
\end{quote}

Setting aside the broader context of enlargements, the issue that Machover discusses arises already at the level of Luxemburg ultrapowers, since, if the Continuum Hypothesis fails, two ultrapowers of the reals may fail to be isomorphic. Machover's point is then that speaking of \textit{the} hyperreal line is an abuse of language and that nothing in the practice of nonstandard analysis guarantees that the structure that is being studied is fixed precisely enough, that is, up to isomorphism. Here, the fact that Luxemburg ultrapowers rely on the Ultrafilter Lemma is relevant once again: since non-principal ultrafilters are abstract objects which can be proved to exist but cannot be explicilty constructed, it is not possible to fix one particular ultrafilter as a \textit{canonical} one, and consequently no Luxemburg ultrapower can claim to be the privileged hyperreal line. In other words, fixing one ultrapower as the hyperreal line seems both necessary, due to the consistency of the existence of non-isomorphic ultrapowers, and impossible without introducing an element of arbitrariness. In \cite{kanovei2013tools}, Kanovei et al. note that Machover's point should be nuanced by two existing results. First, the existence of an ordinal definable model of the hyperreals due to Kanovei and Shelah mentioned above, and second, a result of Morley and Vaught establishing that ultrapowers induced by ultrafilters on a cardinal $\kappa$ such that $2^{<\kappa} = \kappa$ are all isomorphic. I am not certain, however, that this entirely addresses Machover's criticism, as I believe that there is more to the notion of canonicity than the simple fact that a structure is specified up to isomorphism or ``definable'' in a very narrow set-theoretic sense. Indeed, if the only reason to declare the Kanovei-Shelah model or one of the ultrapowers shown by Morley and Vaught to be unique up to isomorphism as \textit{the} hyperreal line is to rebuke Machover's criticism, then this hardly eliminates the charge of arbitrariness that is at the heart of the canonicity objection. Why should we choose the Kanovei-Shelah model, or one of the Morley-Vaught ones, over one another, as the privileged hyperreal line? \\

According to Machover, the lack of a canonical hyperreal line also means that the well-definedness of many concepts from nonstandard analysis ultimately relies on the existence of their standard counterpart. He takes as a paradigmatic example of this the nonstandard definition of continuity of a real-valued function $f$ at a real number $r$. As mentioned in \cref{fhyl}, one of the most appealing features of nonstandard analysis is that infinitesimals can be used in a rigorous way to give an intuitive definition of continuity at a real number $r$: $f$ is continuous at $r$ if and only if for any $^*$real $x$, $f(r)$ is infinitesimally close to $f(r)$ whenever $x$ is infinitesimally close to $r$. In other words, the nonstandard halo of $r$ must be mapped by $f$ to the nonstandard halo of $f(r)$. But, as Machover observes:

\begin{quote}
[W]hat we want to define here is a binary relation between two standard objects, $f$ and $r$; in order to legitimize [the nonstandard definition of continuity] as a definition of this relation, we must make sure that it is independent of the choice of the enlargement. (Otherwise, what is being defined would be a \textit{ternary} relation between $f$, $r$ and the enlargement.) The easiest way---in fact, the only practicable way, as far as I know---to prove this invariance of [the nonstandard definition] is to show that it is equivalent to the standard $\delta-\epsilon$ definition. Therefore, [the nonstandard definition] cannot displace the old standard definition altogether, if one's aim is to achieve proper rigour and methodological correctness. [p.~208]
\end{quote}

Because the halo of a real number is dependent on the ambient hyperreal line or nonstandard universe, the nonstandard definition of continuity must be supplemented with a proof that it is actually independent from the choice of an enlargement. According to Machover, such a proof can in practice only be given by showing the equivalence with the standard definition of continuity. Hence the nonstandard definition owes a conceptual debt to the standard definition, which, according to Machover, means that nonstandard analysis cannot aim at rigorously replacing classical analysis.\\

Machover's example of the definition of continuity is a particularly good point to compare the classical hyperreal line with $\fhyl$. As we have seen in \cref{cont}, a real-valued function $f$ is continuous at a point $r$ if and only if $f$ is $F$-continuous at $r$ for any $F \in \frec$, i.e., $x \in (r)_F$ implies $f(x) \in (f(r))_F$ for any $x \in \RR^\omega$. It seems at first sight that the situation is comparable to the nonstandard one: we first define continuity as a ternary relation between a real-valued function, a real number and a filter $F$, before showing that the relation actually holds independently of the filter $F$ by showing that it is equivalent to the standard definition of continuity. However, in our case, there is a simple argument that allows us to show that a function $f$ is either $F$-continuous at some real number $r$ for every $F \in \frec$ or for no $F \in \frec$. Indeed, we may expand our language $\la$ to a language $\la^*$ with a new binary relation symbol $\simeq$ to represent infinitesimal closeness of two $F$-hyperreals. Formally, this means that we interpret this new symbol $\simeq$ so that for any $a,b \in \RR^\omega$ and any $F\in \frec$, $F \fr a \simeq b$ iff $|||a-b| < \frac{1}{n}|| \in F$ for every $n \in \omega$. It is straightforward to verify that this interpretation of $\simeq$ satisfies the persistence and refinability conditions of possibility semantics. Hence we may now view $\fhyl$ as an $\la^*$-possibility structure, in which $F$-continuity of a function $f$ at some real $r$ is equivalent to $F \fr \forall x (x \simeq r \to f(x) \simeq f(r))$. Similarly to \cref{losx}, \L\'os's Theorem holds for all $\la^*$-formulas. I now claim the following:

\begin{lemma} \label{invacon}
For any function $f: \RR \to \RR$ and any $r \in \RR$, $F_0 \fr \forall x (x \simeq r \to f(x) \simeq f(r))$, or $F_0 \fr \exists x (x \simeq r \me \neg (f(x) \simeq f(r))$. As a consequence, $f$ is either $F$-continuous at $r$ for every $F \in \frec$ or for no $F \in \frec$.
\end{lemma}

\begin{proof}
By persistence, it is immediate to see that if $F_0$ forces that $f$ is continuous at $r$ or that $F$ is not continuous at $r$, then $f$ must either be $F$-continuous at $r$ for every $F \in \frec$, or $F$-discontinuous at $r$ for any $F \in \frec$ respectively. So let us show that either $F_0 \fr \forall x (x \simeq r \to f(x) \simeq f(r))$ or $F_0 \fr \exists x (x \simeq r \me \neg f(x) \simeq f(r))$. Suppose that $F_0 \not \fr \forall x (x \simeq r \to f(x) \simeq f(r))$. Then there is some $F \in \frec$ and some $a \in \RR^\omega$ such that $F \fr a \simeq r$ and $F \fr \neg f(a) \simeq f(r)$. This means that $|||a - r| < \frac{1}{n}|| \in F$ for all $n \in \omega$, and that there is some $m \in \omega$ such that $|||f(a) - f(r)|\geq \frac{1}{m}|| \in F$. Since $F$ is a non-principal filter on $\omega$, this means that for any $i \in \omega$, the set $A_i = \bigcap_{k < i} |||a - r| < \frac{1}{k}|| \cap |||f(a) - f(r)| \geq \frac{1}{m}||$ is infinite. Let $b: \omega \to \RR$ be defined by letting $b(i) = a(j)$, where $j$ is the least $n \in A_i$ such that $n > i$. By construction, we have that $i \in |||b-r| < \frac{1}{n}||$ whenever $n < i$, which means that $F_0 \fr b \simeq r$. Moreover, $|||f(b) - f(r)| \geq \frac{1}{m}|| = \omega$, hence $F_0 \fr \neg f(b) \simeq f(r)$. This completes the proof.
\end{proof}
Note that the proof above does not mention the standard definition of continuity in any way. Indeed, the core of the argument is to show that any counterexample to the $F$-continuity of $f$ for \textit{some} non-principal filter $F$ can be turned into a counterexample to the $F_0$-continuity of $f$. As a consequence, $F$-continuity of a function is a property that is entirely determined by the Fréchet filter, and we could therefore substitute $F$-continuity for the standard definition of continuity without any risk of ambiguity. Interestingly, the argument above also gives a way out of Machover's criticism of the nonstandard definition of continuity. Indeed, as the comparison between $\fhyl$ and the generic approach to nonstandard analysis has made apparent, Luxemburg ultrapowers of $\RR$ are precisely generic models over $\frec$. It is straightforward to verify that the results in \cref{ult} would extend to $\fhyl$ considered as an $\la^*$-structure. In particular, this means that for any ultrafilter $U$ on $\omega$, a function $f$ will be continuous at a real number $r$ according to the ultrapower $\RR_{/U}$ if and only if it is $F$-continuous for some $U \rset F \in \frec$. But this, in connection with \cref{invacon}, straightforwardly implies that the non-standard definition of continuity is independent of the choice of the ultrapower $\RR_{/U}$. This shows that, if one takes Machover's challenge seriously, $\fhyl$ has some significant foundational consequences for nonstandard analysis even in the presence of the Ultrafilter Lemma.\\

What about Machover's broader point regarding the lack of canonicity of the classical hyperreal line? From the purely mathematical perspective, there is a precise way in which $\fhyl$ is canonical in the sense of being characterized up to isomorphism. If we adopt Scott's perspective as outlined in \cref{gen}, $\fhyl$ is a Boolean-valued model of $\RR$. Boolean-valued models are characterized by their base structure and their Boolean algebra of truth values. Of course, $\RR$ is the unique complete ordered field while, as discussed in \cref{scott}, the Boolean algebra of truth-values in $\fhyl$ is the canonical extension of $\Po(\omega)^*$. Since the canonical extension of a Boolean algebra $B$ can always be characterized purely algebraically up to a unique isomorphism fixing $B$ (see for instance \cite{gehrke2013delta1}), characterizing $\fhyl$ uniquely up to isomorphism as a Boolean-valued model reduces to characterizing up to isomorphism the Boolean algebra $\Po(\omega)^*$. This turns out to be a slightly more involved matter. In the presence of the Continuum Hypothesis, this can be done directly, as shown in \cite{vanmill1984}. Without assuming CH, one can still argue in an indirect way as follows. By the Lindenbaum-Tarski-J\'onsson duality between complete atomic Boolean algebras and sets (see \cite[Chap.~6]{Johnstone}), $\Po(\omega)$ is the unique complete Boolean algebra generated by a countable set of atoms. Moreover, a \textit{compact} element in a complete Boolean algebra $B$ is some element $b \in B$ such that for any $X \sset B$, $b \leq \bigjo X$ implies $b \leq \bigjo X'$ for some finite $X' \sset X$. It is straightforward to verify that the compact elements of a Boolean algebra $B$ always form an ideal on $B$, and that if $f : B \to \Po(\omega)$ is a Boolean isomorphism, then $b$ is a compact element of $B$ iff $f(b)$ is a finite subset of $\omega$. Putting things together, this means that $\fhyl$ is the unique $\mathbb{B}$-valued model of the complete ordered field, where $\mathbb{B}$ is the canonical extension of the quotient of a complete Boolean algebra generated by a countable set of atoms modulo its compact elements.

Moreover, as I have mentioned above, there is probably more to the informal notion of canonicity than a mere characterization up to isomorphism. In the case of $\fhyl$, I think it is also possible to argue that the structure we obtain is a \textit{natural} mathematical object to investigate. Indeed, as I have argued at length in the previous sections, $\fhyl$ seems to arise as a common ground for many distinct mathematical projects. It is a strengthening of Laugwitz's attempt to provide rigorous foundations to Leibniz and Cauchy's calculus and to Euler's work with infinite sequences, and at the same time it is a formalization of Tao's attempt to bring nonstandard analysis closer to ordinary mathematical practice. It is a simplification of Palmgren's topos-theoretic approach to constructive nonstandard analysis, as well as a more concrete approach to Boolean-valued models of analysis. Finally, it brings a generic perspective to classical nonstandard analysis, allowing Luxemburg ultrapowers to be seen as generic models in the sense of forcing. 

As such, $\fhyl$ can also be seen as providing both a diagnosis and a cure to the canonicity problem identified by Machover. If we agree with Laugwitz that the core idea of the infinitesimal method is that the properties of sequences ``in the limit'' are determined by the properties of all but finitely many of their values, then only the cofinite sets of $\omega$ can be deemed as large enough to determine truth in the hyperreal line. As Tarskian semantics requires an ultrafilter for the Transfer Principle to apply, Robinsonian hyperreals force us to select some infinite and co-infinite sets to count as ``large enough'' as well, thus introducing an element of seemingly unavoidable arbitrariness. By contrast, the dynamic character of possibility semantics allows us to consider all possible such choices at once and to restore the Transfer Principle at the only canonical stage, that of the Fréchet filter.

\section{Conclusion}
In this paper, I have used possibility semantics for first-order logic to define the $F$-hyperreal line $\fhyl$ as an alternative to the classical ultrapower approach to nonstandard analysis. As I have argued, $\fhyl$ shares many of the technical advantages of the classical hyperreal line, arguably coming closer to it than the asymptotic, dynamic and generic approaches I have discussed. At the same time, its more constructive character makes it an attractive option from a foundational and methodological standpoint. I therefore hope to have convinced the reader that the semi-constructive approach to nonstandard analysis I have sketched here is a mathematically natural and philosophically rich alternative to explore. 

The work presented here does not exhaust the ways in which possibility semantics may interact with nonstandard analysis. For one, I have only discussed an alternative to the use of countably saturated ultrapowers, but nonstandard analysis is a much wider field. In particular, it routinely studies a larger class of structures, such as enlargements of the standard universe of sets, and it remains to be seen whether one could also develop a satisfactory alternative to such structures in a semi-constructive setting. Similarly, I have only touched on the idea that the $F$-hyperreal line is a straightforward attempt to capture an intuition about the properties of sequences being determined by their values ``almost everywhere''. Whether this idea is as historically significant as hinted by Laugwitz, and whether it could motivate a proper \textit{conception} of the continuum, in the sense of \cite{feferman2009conceptions}, will have to be explored in future work.

\section*{Acknowledgments}
I would like to thank Wes Holliday and Anna Bellomo for their helpful comments on earlier versions of this manuscript.

\appendix

\section{} \label{app}

In this appendix, I show how to define an analogue of the standard part function on $\fhyl$. As mentioned in \cref{fhyl}, it is not possible to define an actual function $f : \RR^\omega \to \RR$ that would map any finite $F$-hyperreal to its standard part, on pain of arbitrariness or inconsistency. Nonetheless, it is possible to define such a function \textit{internally}, meaning that one may extend the language of $\fhyl$ so as to include a binary relation symbol $st(x,y)$, to be interpreted as ``$x$ is the standard part of $y$''. One can then show that any viewpoint $F \in \frec$ forces that any finite $F$-hyperreal has a unique standard part. The first step is to extend the language introduced in \cref{canon} to include a \textit{standardness} predicate:

\begin{definition}
Let $\la^\dagger$ be the language $\la^*$ augmented with a unary predicate symbol $S(x)$ and a binary relation symbol $st(x,y)$. We extend the interpretation function $\itp$ to $\la^\dagger$ as follows:
\begin{itemize}
    \item For any $a \in \fhyl$ and any $F \in \frec$, $a \in \itp(F,S)$ iff $|| \bigjo_{r \in S} a = \delta(r)|| \in F$ for some finite subset $S$ of $\RR$.
    \item For any $a, b \in \fhyl$ and any $F \in \frec$, $(a,b) \in \itp(F,st)$ iff $F \fr S(a) \me a \simeq b$.
\end{itemize}

It is straightforward to verify that the interpretation of $S(x)$ and $st(x,y)$ satisfies the conditions of \cref{def1}. We may now show that $st(x,y)$ is interpreted as a total function on finite $F$-hyperreals in $\fhyl$:

\begin{lemma}
For any $F \in \frec$, $$F\fr \forall y (\exists z_1 \exists z_2 (S(z_1) \me S(z_2) \me z_1 < y < z_2) \to \exists ! x (st(x,y)).$$
\end{lemma}

\begin{proof}
Fix some $F \in \frec$ and some $b \in \RR^\omega$ such that $F \fr z_1 < b < z_2$ for some $F$-standard $z_1, z_2$ which we may assume to be real numbers. Let us first establish that $F \fr \exists x (st(x,b))$. By \cref{semdef}, we must show that for any $G \rset F$, there is $H \rset G$ and $a \in \RR^\omega$ such that $H \fr st(a,b)$. Fix such a filter $G$. We consider the following two cases:
\begin{itemize}
    \item Case $1$: For some $r \in \RR$, $||b \neq \delta(r)|| \notin G$. Then letting $H$ be the filter generated by $G \cup ||b = \delta(r)||$, we have that $H \fr b = \delta(r)$. Since $\delta(r)$ is clearly $H$-standard, it follows that $H \fr st(\delta(r),b)$.
    \item Case $2$: For all $r \in \RR$, $||b \neq \delta(r)|| \in G$. Then let $A$ be the set $\{r \in \RR : G \not\fr b \leq \delta(r)\}$. Note that $z_1 \in A$ and that $z_2$ is an upper bound of $A$, so by the completeness of the real line $A$ has a least upper bound $r$. Now we distinguish again two cases:
    \begin{itemize}
        \item Case $2.1$: $G \fr b < r$. Since $r$ is the least upper bound of the set $A$, we must have that $G \not\fr b \leq \delta(r - \frac{1}{n})$ for any $n \in \mathbb{N}$, since otherwise $r - \frac{1}{n}$ would be an upper bound of $A$. But this means that $||b \leq \delta(r - \frac{1}{n})|| \notin G$ for any $n \in \mathbb{N}$, and thus the set $G \cup \{||\delta(r)-y < \frac{1}{n}|| : n \in \mathbb{N}\}$ generates a filter $H$. By construction, $H \fr \delta(r) \simeq b$, and therefore $H \fr st(\delta(r),b)$.
        \item Case $2.2$: $G \not\fr b < r$. For any $n \in \mathbb{N}$, let $B_n = ||0 < y - \delta(r) < \frac{1}{n}||$. I claim that the set $G \cup \{B_n : n \in \mathbb{N}\}$ generates a filter. To show this, it is enough to show that $\omega \setminus B_n \notin G$ for any $n \in \mathbb{N}$. Suppose towards a contradiction that $\omega \setminus B_n \in G$ for some $n \in \mathbb{N}$. This means that $G \fr b \leq \delta(r) \jo \delta(r + \frac{1}{n}) \leq b$. Since $r$ is an upper bound of $A$, $r+\frac{1}{n} \notin A$, which means that $G \fr b \leq \delta(r + \frac{1}{n})$. By distributivity, this means that $$G \fr (b \leq \delta(r) \me b \leq \delta(r + \frac{1}{n})) \jo (\delta(r + \frac{1}{n}) \leq b \me b \leq \delta(r + \frac{1}{n})),$$ which is equivalent to $$G \fr b < \delta(r) \jo b = \delta(r) \jo b = \delta(r + \frac{1}{n}).$$ 
        \noindent Since $||b \neq \delta(s)|| \in G$ for all $s \in \RR$, this implies that $G \fr b < \delta(r)$, a contradiction. Thus $G \cup \{B_n : n \in \mathbb{N}\}$ generates a filter $H \in \frec$. By construction, $H \fr \delta(r) \simeq b$, and therefore $H \fr st(\delta(r),b)$.
    
    \end{itemize}
\end{itemize}
This completes the proof that $F \fr \exists x(st(x,b))$. For uniqueness, it suffices to show that $F \fr st(a_1,b) \me st(a_2,b)$ implies $F \fr a_1 = a_2$ for any $a_1,a_2 \in \RR^\omega$ and $F \in \frec$. From $F \fr st(a_1,b)$, we have that $||a_1 = \delta(r_1)|| \in F$ for some $r_1 \in F$ and that $F \fr a_1 \simeq b$. But this means that $|||\delta(r_1)-b| < \frac{1}{n}|| \in F$ for any $n \in \mathbb{N}$. Similarly, $|||\delta(r_2)-b| < \frac{1}{n}|| \in F$ for any $n \in \mathbb{N}$. By the triangle inequality, $|\delta(r_1)(i) - \delta(r_2)(i)| \leq |\delta(r_1)(i) - y(i)| + |\delta(r_2)(i) - y(i)|$ for any $i \in \omega$, hence for any $n \in \mathbb{N}$, $|||\delta(r_1)-y| < \frac{1}{2n}|| \cap |||\delta(r_2)-y| < \frac{1}{2n}|| \sset |||\delta(r_1)-\delta(r_2)| < \frac{1}{n}||$, which implies that $|||\delta(r_1)-\delta(r_2)| < \frac{1}{n}|| \in F$ and is thus non-empty. But this implies at once that $|r_1 - r_2| < \frac{1}{n}$ for any $n \in \mathbb{N}$, and therefore $r_1 = r_2$. Thus $||a_1 = r_1|| \cap ||a_2 = r_2|| \sset ||a_1 = a_2||$, and $F \fr a_1 = a_2$.
\end{proof}
\end{definition}

\printbibliography

\end{document}